\documentclass[12pt]{amsart}
\usepackage{amssymb}
\usepackage{amsmath}
\usepackage{mathrsfs}
\usepackage{enumerate}
\usepackage{amsthm}
\usepackage{cancel}
\usepackage{graphicx}
\newtheorem{theorem}{Theorem}[section]
\newtheorem{lemma}[theorem]{Lemma}

\theoremstyle{definition}
\newtheorem{problem}{Problem}

\newtheorem{remark}[theorem]{Remark}

\numberwithin{equation}{section}
\usepackage{color}
\usepackage{hyperref}
\hypersetup{
    colorlinks=true, 
    linkcolor=blue,  
}
\mathchardef\hyphen="2D
\begin{document}
\allowdisplaybreaks
\title{Mixture and separation of log-concave measures}
\author{March~T.~Boedihardjo and Yao Xie}
\address{Department of Mathematics, Michigan State University, East Lansing, MI 48824}
\address{H. Milton Stewart School of Industrial and Systems Engineering, Georgia Institute of Technology, Atlanta, GA, 30332}
\email{boedihar@msu.edu, yao.xie@isye.gatech.edu}
\begin{abstract}
For a two-component log-concave mixture model, we investigate the extent to which the weight, mean, and covariance of each component distribution can be accurately recovered when given sufficient samples of the mixture distributions. One fundamental obstruction is that the mixture distribution itself could sometimes be log-concave, and in this case, accurate recovery is impossible. In this paper, we identify two regimes, one where the mixture distribution itself could be log-concave and another one where the mixture distribution is never log-concave, and one can always separate the two component distributions using a quadratic classifier.
\end{abstract}
\subjclass[2020]{}
\maketitle

\section{Introduction}
We study the problem of recovering the components of a given mixture model. For a two-component Gaussian mixture model, accurate recovery is known to be possible as long as the two components are non-identical and one is given sufficient samples of the mixture \cite{kmv}. For a general log-concave mixture model, although some initial work was done two decades ago in \cite{ksv}, little work has been done since then. This is due to the lack of understanding of the behavior of general log-concave measures. However, recent breakthroughs on the KLS conjecture \cite{kls} and the thin shell conjecture \cite{thinshell} provide powerful tools for studying log-concave measures. These tools enable us to revisit the recovery problem for a two-component log-concave mixture model and identify the regimes where accurate recovery might be possible.

Before going into the details of the problems, we first present some basic setups. Throughout this paper, $n\in\mathbb{N}$ with $n\geq 2$, $\{e_{1},\ldots,e_{n}\}$ is the canonical basis for $\mathbb{R}^{n}$ and $\langle\,,\,\rangle$ denotes the canonical inner product on $\mathbb{R}^{n}$. For $w\in\mathbb{R}^{n}$, we denote by $\|w\|$ the Euclidean norm of $w$. We denote by $\mathbb{R}^{n\times n}$ the set of all $n\times n$ real matrices. For $A\in\mathbb{R}^{n\times n}$, the Frobenius norm of $A$ is denoted by $\|A\|_{\mathrm{F}}$, the spectral norm of $A$ is denoted by $\|A\|_{\mathrm{op}}$, the range of $A$ is denoted by $\mathrm{ran}(A)$, the kernel of $A$ is denoted by $\mathrm{ker}(A)$, and the $(i,j)$-entry of $A$ is denoted by $A_{i,j}$ where $i,j\in\{1,\ldots,n\}$. For $a_{1},\ldots,a_{n}\in\mathbb{R}$, we denote by $(a_{i})_{1\leq i\leq n}\in\mathbb{R}^{n}$ the vector in $\mathbb{R}^{n}$ with entries $a_{1},\ldots,a_{n}$, and we denote by $\mathrm{diag}((a_{i})_{1\leq i\leq n})$ the $n\times n$ diagonal matrix with diagonal entries $a_{1},\ldots,a_{n}$. The identity matrix is denoted by $I$.

If $W$ is a random variable in $\mathbb{R}$, we denote by $\mathrm{Var}(W)$ the variance of $W$.

If $\mu$ is a probability measure on $\mathbb{R}^{n}$ and $F:\mathbb{R}^{n}\to\mathbb{R}^{n}$ is a measurable function, then $F_{\#}\mu$ denotes the pushforward measure of $\mu$ by $F$, i.e., if $X$ is a random vector in $\mathbb{R}^{n}$ distributed according to $\mu$, then $F_{\#}\mu$ is the distribution of the random vector $F(X)$.

Let $\mathcal{P}_{2}(\mathbb{R}^{n})$ be the set of all probability measures $\mu$ on $\mathbb{R}^{n}$ with $\int_{\mathbb{R}^{n}}\|x\|^{2}\,d\mu(x)<\infty$. For $\mu\in\mathcal{P}_{2}(\mathbb{R}^{n})$, the {\it mean} of $\mu$ is the vector
\[z(\mu):=\int_{\mathbb{R}^{n}}x\,d\mu(x),\]
the {\it second moment} of $\mu$ is the $n\times n$ positive semidefinite matrix
\[B_{\mu}:=\int_{\mathbb{R}^{n}}xx^{T}\,d\mu(x),\]
where $x^{T}$ is the transpose of $x$, and the {\it covariance} of $\mu$ is the positive semidefinite matrix 
\[\mathrm{Cov}(\mu):=B_{\mu}-z(\mu)z(\mu)^{T}.\]

For $z\in\mathbb{R}^{n}$ and a positive semidefinite matrix $\Sigma\in\mathbb{R}^{n\times n}$, the normal distribution with mean $z$ and covariance $\Sigma$ is denoted by $\mathcal{N}(z,\Sigma)$.

Suppose that $\mu_{1}$ and $\mu_{2}$ are normal distributions on $\mathbb{R}^{n}$ with unknown means and covariances, $\mu_{1}\neq\mu_{2}$, and $0<\alpha<1$ is unknown. Given sufficient samples of the Gaussian mixture $\alpha\mu_{1}+(1-\alpha)\mu_{2}$, there is an algorithm \cite[Theorem 1]{kmv} that accurately recovers, with high probability, the weight $\alpha$ and the means and covariances of the component distributions $\mu_{1}$ and $\mu_{2}$ (up to permutation of the two components). Moreover, there is a classifier \cite[Corollary 3]{kmv} that separates the two component distributions $\mu_{1}$ and $\mu_{2}$.

In this paper, we investigate the extent to which these results hold for a general two-component log-concave mixture model.
\begin{problem}\label{introproblem}
Suppose that $\mu_{1}$ and $\mu_{2}$ are unknown log-concave probability measures on $\mathbb{R}^{n}$ and $0<\alpha<1$ is a unknown. Given sufficient samples of the mixture distribution $\alpha\mu_{1}+(1-\alpha)\mu_{2}$, is it possible to accurately recover the weight $\alpha$ and the means and covariances of $\mu_{1}$ and $\mu_{2}$ and to generate a classifier that separates $\mu_{1}$ and $\mu_{2}$?
\end{problem}
Assume that the weight $\alpha$ is bounded away from $0$ or $1$. When the means $z(\mu_{1}),z(\mu_{2})$ of the two component distributions are relatively far apart so that
\[\|z(\mu_{1})-z(\mu_{2})\|\gg(\ln n)\cdot(\|\mathrm{Cov}(\mu_{1})\|_{\mathrm{op}}+\|\mathrm{Cov}(\mu_{2})\|_{\mathrm{op}}),\]
there is an algorithm \cite[Theorem 2]{ksv} that accurately recovers, with high probability, the weight $\alpha$ and the means and covariances of $\mu_{1}$ and $\mu_{2}$ (up to permutation of the two components). Moreover, one can correctly classify the sample points into those that come from $\mu_{1}$ and those that come from $\mu_{2}$ \cite[Theorem 3]{ksv}. The idea is that after applying PCA on the sample points, there will be two clusters that are far apart, one consisting of points that come from $\mu_{1}$ and the other one consisting of points that come from $\mu_{2}$. After identifying these two clusters, one can then compute the mean and covariance for the sample points associated with each cluster. 

However, in the absence of the assumption that $z(\mu_{1})$ and $z(\mu_{2})$ are relatively far apart (e.g., when $z(\mu_{1})=z(\mu_{2})$), the strategy outlined in the previous paragraph does not work in general, since the two clusters obtained after applying PCA are no longer necessarily far apart.

In fact, there is a fundamental obstruction, which, if it happens, makes accurate recovery impossible. When the two-component mixture $\alpha\mu_{1}+(1-\alpha)\mu_{2}$ in Problem \ref{introproblem} is itself log-concave, it becomes a one-component distribution by itself, and we can write $\alpha\mu_{1}+(1-\alpha)\mu_{2}$ as another mixture decomposition:
\[\alpha\mu_{1}+(1-\alpha)\mu_{2}=\beta\cdot(\alpha\mu_{1}+(1-\alpha)\mu_{2})+(1-\beta)\cdot(\alpha\mu_{1}+(1-\alpha)\mu_{2}),\]
for every $0\leq\beta\leq 1$, and each component $\alpha\mu_{1}+(1-\alpha)\mu_{2}$ in this new decomposition is log-concave. In this case, it is impossible to accurately recover the weight $\alpha$ from samples of the mixture distribution $\alpha\mu_{1}+(1-\alpha)\mu_{2}$, unless we have additional knowledge about $\mu_{1},\mu_{2}$ (e.g., $\mu_{1},\mu_{2}$ being normal distributions).

Note that the obstruction explained in the previous paragraph does not occur in the context of Gaussian mixture models, since a mixture of non-identical normal distributions is never a normal distribution.

In this paper, we identify two regimes, one where the above obstruction could occur and another one where the above obstruction never occurs. Let $z_{1}=z(\mu_{1})$, $z_{2}=z(\mu_{2})$, $\Sigma_{1}=\mathrm{Cov}(\mu_{1})$, $\Sigma_{2}=\mathrm{Cov}(\mu_{2})$.

{\bf Regime 1:} This regime is the case where $z_{1}-z_{2}\in\mathrm{ran}(\Sigma_{1}+\Sigma_{2})$ and
\[\|(\Sigma_{1}+\Sigma_{2})^{-1/2}(\Sigma_{1}-\Sigma_{2})(\Sigma_{1}+\Sigma_{2})^{-1/2}\|_{\mathrm{F}}+\|(\Sigma_{1}+\Sigma_{2})^{-1/2}(z_{1}-z_{2})\|^{2}\leq c,\]
where $c>0$ is a small absolute constant.

In this regime, we show that the mixture distribution $\alpha\mu_{1}+(1-\alpha)\mu_{2}$ could sometimes be log-concave, and as explained above, in this case, accurate recovery of $\alpha$ is impossible unless we have additional knowledge about $\mu_{1},\mu_{2}$.

{\bf Regime 2:} This regime consists of two cases:\\
(i) $z_{1}-z_{2}\notin\mathrm{ran}(\Sigma_{1}+\Sigma_{2})$, or\\
(ii) $z_{1}-z_{2}\in\mathrm{ran}(\Sigma_{1}+\Sigma_{2})$ and
\[\|(\Sigma_{1}+\Sigma_{2})^{-1/2}(\Sigma_{1}-\Sigma_{2})(\Sigma_{1}+\Sigma_{2})^{-1/2}\|_{\mathrm{F}}+\|(\Sigma_{1}+\Sigma_{2})^{-1/2}(z_{1}-z_{2})\|^{2}\gg\sqrt{\ln n}.\]

In this regime, we show that the mixture distribution $\alpha\mu_{1}+(1-\alpha)\mu_{2}$ is never log-concave for $\alpha$ bounded away from $0$ or $1$. Moreover, there is a quadratic set $L\subset\mathbb{R}^{n}$ that depends only on $z_{1},z_{2},\Sigma_{1},\Sigma_{2}$ such that
\[\mu_{1}(L)\geq 1-\delta\quad\text{and}\quad\mu_{2}(\mathbb{R}^{n}\backslash L)\geq 1-\delta,\]
where $\delta>0$ is some desired threshold, i.e., $\mu_{1}$ and $\mu_{2}$ are ``almost disjointly supported." The set $L$ can be thought of as a classifier separating the two probability measures $\mu_{1}$ and $\mu_{2}$. These two results suggest that in this regime, there is some hope to accurately recover the unknown weight $\alpha$ and the unknown means and covariances $z_{1},z_{2},\Sigma_{1},\Sigma_{2}$ when given sufficient samples of the mixture distribution $\alpha\mu_{1}+(1-\alpha)\mu_{2}$.

The conditions in Regime 1 and Regime 2 may, at first, look a bit technical. So let's focus on the most interesting case where $z_{1}=z_{2}$, since, as mentioned above, the case where $z_{1}$ and $z_{2}$ are far apart has already been done in \cite{ksv}. When $z_{1}=z_{2}$, Regime 1 is the case where
\[\|(\Sigma_{1}+\Sigma_{2})^{-1/2}(\Sigma_{1}-\Sigma_{2})(\Sigma_{1}+\Sigma_{2})^{-1/2}\|_{\mathrm{F}}\leq c,\]
whereas Regime 2 is the case where
\[\|(\Sigma_{1}+\Sigma_{2})^{-1/2}(\Sigma_{1}-\Sigma_{2})(\Sigma_{1}+\Sigma_{2})^{-1/2}\|_{\mathrm{F}}\gg\sqrt{\ln n}.\]

For example, if $\Sigma_{1}=bI$ and $\Sigma_{2}=I$ where $b>0$, then Regime 1 is the case where $|\frac{b-1}{b+1}|\leq\frac{c}{\sqrt{n}}$, whereas Regime 2 is the case where $|\frac{b-1}{b+1}|\gg\sqrt{\frac{\ln n}{n}}$. When $n$ is large, Regime 1 is equivalent to $|b-1|\leq\frac{\widetilde{c}}{\sqrt{n}}$, where $\widetilde{c}>0$ is a small absolute constant, whereas Regime 2 is equivalent to $|b-1|\gg\sqrt{\frac{\ln n}{n}}$. This is because if $b\leq 2$, then $\frac{1}{3}|b-1|\leq|\frac{b-1}{b+1}|\leq|b-1|$; and if $b\geq 2$, then $|\frac{b-1}{b+1}|\geq\frac{1}{3}$ and $b\geq 2$ always belongs to Regime 2. As we can see, Regime 1 requires $b$ to be really close to $1$, whereas Regime 2 only requires $b$ to be slightly different from 1.

More generally, when the matrices $\Sigma_{1}$ and $\Sigma_{2}$ have high stable ranks, Regime 1 requires $\Sigma_{1}$ and $\Sigma_{2}$ to be really close to each other, whereas Regime 2 only requires $\Sigma_{1}$ and $\Sigma_{2}$ to be slightly different from each other. Indeed, we always have
\[\|(\Sigma_{1}+\Sigma_{2})^{-1/2}(\Sigma_{1}-\Sigma_{2})(\Sigma_{1}+\Sigma_{2})^{-1/2}\|_{\mathrm{F}}\geq\frac{\|\Sigma_{1}-\Sigma_{2}\|_{\mathrm{F}}}{\|\Sigma_{1}+\Sigma_{2}\|_{\mathrm{op}}}\geq\frac{\|\Sigma_{1}-\Sigma_{2}\|_{\mathrm{F}}}{\|\Sigma_{1}\|_{\mathrm{op}}+\|\Sigma_{2}\|_{\mathrm{op}}},\]
and so Regime 2 includes the case where
\[\|\Sigma_{1}-\Sigma_{2}\|_{\mathrm{F}}\gg\sqrt{\ln n}\cdot(\|\Sigma_{1}\|_{\mathrm{op}}+\|\Sigma_{2}\|_{\mathrm{op}}).\]
When the stable ranks $(\frac{\|\Sigma_{1}\|_{\mathrm{F}}}{\|\Sigma_{1}\|_{\mathrm{op}}})^{2}$ and $(\frac{\|\Sigma_{2}\|_{\mathrm{F}}}{\|\Sigma_{2}\|_{\mathrm{op}}})^{2}$ are of order $n$, this condition is equivalent to
\[\|\Sigma_{1}-\Sigma_{2}\|_{\mathrm{F}}\gg\sqrt{\frac{\ln n}{n}}\cdot(\|\Sigma_{1}\|_{\mathrm{F}}+\|\Sigma_{2}\|_{\mathrm{F}}).\]
\subsection{Can two be one?}
In this subsection, we address the issue of the mixture distribution itself being log-concave by trying to identify the regimes where this issue could occur and regimes where this issue never occurs.
\begin{problem}\label{probidentifyset}
Fix $n\in\mathbb{N}$ and $0<\alpha<1$. Identify the set $\Lambda_{n,\alpha}$ of all tuples $(z_{1},z_{2},\Sigma_{1},\Sigma_{2})\in\mathbb{R}^{n}\times\mathbb{R}^{n}\times\mathbb{R}^{n\times n}\times\mathbb{R}^{n\times n}$ for which there exist log-concave probability measures $\mu_{1}$ and $\mu_{2}$ on $\mathbb{R}^{n}$ such that $\mu_{j}$ has mean $z_{j}$ and covariance $\Sigma_{j}$, for $j=1,2$, and $\alpha\mu_{1}+(1-\alpha)\mu_{2}$ is log-concave.
\end{problem}
The first main result of this paper is a partial answer to Problem \ref{probidentifyset} where we provide a sufficient condition and a necessary condition for a given tuple $(z_{1},z_{2},\Sigma_{1},\Sigma_{2})$ to be in $\Lambda_{n,\alpha}$. While $\Sigma_{1}+\Sigma_{2}$ is not always invertible, we can always define the operator $(\Sigma_{1}+\Sigma_{2})^{-1/2}$. See Appendix \ref{inversesection}.
\begin{theorem}\label{firstmain}
There are absolute constants $C,c>0$ such that the following hold. Suppose that $z_{1},z_{2}\in\mathbb{R}^{n}$ and $\Sigma_{1},\Sigma_{2}\in\mathbb{R}^{n\times n}$ are positive semidefinite matrices.
\begin{enumerate}[(1)]
\item If $z_{1}-z_{2}\in\mathrm{ran}(\Sigma_{1}+\Sigma_{2})$ and
\[\|(\Sigma_{1}+\Sigma_{2})^{-1/2}(\Sigma_{1}-\Sigma_{2})(\Sigma_{1}+\Sigma_{2})^{-1/2}\|_{\mathrm{F}}+\|(\Sigma_{1}+\Sigma_{2})^{-1/2}(z_{1}-z_{2})\|^{2}\leq c,\]
then there exist probability measures $\mu_{1}$ and $\mu_{2}$ on $\mathbb{R}^{n}$ such that $\mu_{j}$ has mean $z_{j}$ and covariance $\Sigma_{j}$, for $j=1,2$, and $\alpha\mu_{1}+(1-\alpha)\mu_{2}$ is log-concave for all $\alpha\in[0,1]$.
\item Conversely, for any fixed $\alpha\in(0,1)$, if there exist probability measures $\mu_{1}$ and $\mu_{2}$ on $\mathbb{R}^{n}$ such that $\mu_{j}$ has mean $z_{j}$ and covariance $\Sigma_{j}$, for $j=1,2$, and $\alpha\mu_{1}+(1-\alpha)\mu_{2}$ is log-concave, then $z_{1}-z_{2}\in\mathrm{ran}(\Sigma_{1}+\Sigma_{2})$ and
\[\|(\Sigma_{1}+\Sigma_{2})^{-1/2}(\Sigma_{1}-\Sigma_{2})(\Sigma_{1}+\Sigma_{2})^{-1/2}\|_{\mathrm{F}}+\|(\Sigma_{1}+\Sigma_{2})^{-1/2}(z_{1}-z_{2})\|^{2}\leq\frac{C\sqrt{\ln n}}{\alpha(1-\alpha)}.\]
\end{enumerate}
\end{theorem}
Note that Statement (2) in Theorem \ref{firstmain} does not require $\mu_{1},\mu_{2}$ to be log-concave.

\subsection{Quadratic classifier}
The second main result of this paper provides a sufficient condition for the existence of a quadratic classifier separating two log-concave probability measures $\mu_{1}$ and $\mu_{2}$.

A {\it quadratic set} $L\subset\mathbb{R}^{n}$ is a set of the form
\[L=\{x\in\mathbb{R}^{n}:\,\langle Ax,x\rangle+\langle x,w\rangle+s>0\},\]
for some $A\in\mathbb{R}^{n\times n}$, $w\in\mathbb{R}^{n}$, $s\in\mathbb{R}$.
\begin{theorem}\label{secondmain}
Let $0<\delta<\frac{1}{2}$. Suppose that $z_{1},z_{2}\in\mathbb{R}^{n}$ and $\Sigma_{1},\Sigma_{2}\in\mathbb{R}^{n\times n}$ are positive semidefinite matrices. If $z_{1}-z_{2}\notin\mathrm{ran}(\Sigma_{1}+\Sigma_{2})$, or if $z_{1}-z_{2}\in\mathrm{ran}(\Sigma_{1}+\Sigma_{2})$ and
\begin{equation}\label{secondmaineq1}
\|(\Sigma_{1}+\Sigma_{2})^{-1/2}(\Sigma_{1}-\Sigma_{2})(\Sigma_{1}+\Sigma_{2})^{-1/2}\|_{\mathrm{F}}+\|(\Sigma_{1}+\Sigma_{2})^{-1/2}(z_{1}-z_{2})\|^{2}\geq C\left(\ln\frac{1}{\delta}\right)^{2}\sqrt{\ln n},
\end{equation}
where $C>0$ is a large enough absolute constant, then there exists a quadratic set $L\subset\mathbb{R}^{n}$ that depends only on $z_{1},z_{2},\Sigma_{1},\Sigma_{2}$ such that
\[\mu_{1}(L)\geq 1-\delta\quad\text{and}\quad\mu_{2}(\mathbb{R}^{n}\backslash L)\geq 1-\delta,\]
for all log-concave probability measures $\mu_{1}$ and $\mu_{2}$ on $\mathbb{R}^{n}$ such that $\mu_{j}$ has mean $z_{j}$ and covariance $\Sigma_{j}$ for $j=1,2$.
\end{theorem}
\subsection{The logarithmic factor}
The key ingredient used in the proof of Statement (2) in Theorem \ref{firstmain} and the proof of Theorem \ref{secondmain} is the non-isotropic thin shell bound Lemma \ref{nsthinshell} below. Unfortunately, this result involves a $\sqrt{\ln n}$ factor, and it is only because of this that there are $\sqrt{\ln n}$ factors in Statement (2) in Theorem \ref{firstmain} and Theorem \ref{secondmain}. It is not known if the $\sqrt{\ln n}$ factor in Lemma \ref{nsthinshell} can be removed. It is, however, known that the thin shell conjecture \cite{thinshell} holds without any logarithmic factor.
\subsection{Third regime}
There is an important case not covered by Regime 1 or Regime 2. Suppose that the dimension $n=O(1)$ is not too large, $z_{1}=z_{2}=0$ and $\Sigma_{1}=I$ and $\|\Sigma_{2}\|_{\mathrm{op}}=o(1)$ is small. In this case,
\[(\Sigma_{1}+\Sigma_{2})^{-1/2}(\Sigma_{1}-\Sigma_{2})(\Sigma_{1}+\Sigma_{2})^{-1/2}\approx I\]
has Frobenius norm of order $1$, and so the assumption of Theorem \ref{secondmain} is not satisfied. However, we believe that in this case, some version of the conclusion of Theorem \ref{secondmain} may still hold with the $\delta$ being some polynomial of $\frac{1}{\|\Sigma_{2}\|}$. For example, suppose that $\mu_{1}$ is the uniform probability distribution on the Euclidean ball $B(0,a)$ in $\mathbb{R}^{n}$ centered at $0$ with radius $a$, where $a>0$ is such that $\mu_{1}$ has covariance $\Sigma_{1}=I$. Suppose that $\mu_{2}$ is the uniform probability distribution on the Euclidean ball $B(0,ba)$ in $\mathbb{R}^{n}$ centered at $0$ with radius $ba$, where $0<b<1$ is small. Then $\mu_{2}$ has covariance $\Sigma_{2}=b^{2}I$. Taking $L=\mathbb{R}^{n}\backslash B(0,ba)$, we have $\mu_{1}(L)=1-b^{n}$ and $\mu_{2}(\mathbb{R}^{n}\backslash L)=\mu_{2}(B(0,ba))=1$.

We also believe that more generally, if $z_{1}=z_{2}=0$ and $\|\Sigma_{2}^{-1/2}\Sigma_{1}\Sigma_{2}^{-1/2}-I\|_{\mathrm{F}}$ is large, then the conclusion of Theorem \ref{secondmain} still holds, but even so, it is not clear what is the optimal dependence of $\delta>0$ on $\Sigma_{1}$ and $\Sigma_{2}$.

\subsection{Organization of the paper}
The rest of this paper is organized as follows. In Section \ref{firstmainp1proofsection}, we prove Statement (1) in Theorem \ref{firstmain}. In Section \ref{nsthinshellsection}, we derive a non-isotropic thin shell bound using \cite{kls}. This is the key ingredient for the proofs in the next two sections. In Section \ref{firstmainp2proofsection}, we prove Statement (2) in Theorem \ref{firstmain}. In Section \ref{secondmainproofsection}, we prove Theorem \ref{secondmain}.

\section{Proof of the first main result part (1)}\label{firstmainp1proofsection}
In this section, we prove Statement (1) in Theorem \ref{firstmain}. The idea of the proof is as follows.

First, since log-concavity is invariant under translation, without loss of generality, we can assume that $z_{2}=0$. Second, since log-concavity is also invariant under pushforward by any linear transformation, we can further reduce to the case where $\Sigma_{1}=I$ is the identity matrix and $\Sigma_{2}$ is a diagonal matrix. The key lemma that enables this reduction step is Lemma \ref{changeofbasislemma}.

So now, we need to show that if $z\in\mathbb{R}^{n}$ with $\|z\|\leq c_{3}$ and $D\in\mathbb{R}^{n\times n}$ is a diagonal matrix with $\|D-I\|_{\mathrm{F}}\leq c_{3}$, where $c_{3}>0$ is a small enough absolute constant, then there exist a probability measure $\mu_{1}$ on $\mathbb{R}^{n}$ with mean $z$ and covariance $I$ and a probability measure $\mu_{2}$ on $\mathbb{R}^{n}$ with mean $0$ and covariance $D$ such that $\alpha\mu_{1}+(1-\alpha)\mu_{2}$ is log-concave for all $\alpha\in[0,1]$. This is exactly the statement of Lemma \ref{firstmainp1diagonal} below.

But how do we construct these two measures $\mu_{1},\mu_{2}$ in Lemma \ref{firstmainp1diagonal}? The most natural attempt is to take $\mu_{1}=\mathcal{N}(z,I)$ and $\mu_{2}=\mathcal{N}(0,D)$ to be normal distributions. Unfortunately, the mixture of these two normal distributions is not necessarily log-concave. The mixture, however, is log-concave when restricted to the hypercube $[-1,1]^{n}$. But once we restrict the normal distributions $\mathcal{N}(z,I)$ and $\mathcal{N}(0,D)$ to the hypercube $[-1,1]^{n}$, the means and covariances are no longer $z,0,I,D$.

The measures $\mu_{1},\mu_{2}$ in Lemma \ref{firstmainp1diagonal} are constructed as follows. We let $\nu_{1}$ to be the restriction of the normal distribution $\mathcal{N}(w,I)$ to the hypercube $[-1,1]^{n}$, where $w\in\mathbb{R}^{n}$ is a suitably chosen vector, and we let $\nu_{2}$ to be the restriction of $\mathcal{N}(0,E)$, where $E\in\mathbb{R}^{n\times n}$ is a suitably chosen diagonal matrix. Then $\mu_{1}$ and $\mu_{2}$ are the pushforward measures of $\nu_{1}$ and $\nu_{2}$, respectively, by some suitably chosen diagonal matrix $G$.

This section is outlined as follows.

In Subsection \ref{firstmainp1proofsubsection1}, we prove Lemma \ref{Fconcave}, which says that if $w\in\mathbb{R}^{n}$ is close enough $0$ and $E\in\mathbb{R}^{n\times n}$ is a diagonal matrix close enough to $I$, then any weighted average of the two functions $x\mapsto e^{-\|x-w\|^{2}}$ and $x\mapsto e^{-\langle Ex,x\rangle}$ is log-concave on the hypercube $[-1,1]^{n}$. This result is essential for proving the log-concavity of the mixture of the measures we construct in the next subsection.

In Subsection \ref{firstmainp1proofsubsection2}, we prove Lemma \ref{firstmainp1diagonal}, which says that Statement (1) in Theorem \ref{firstmain} holds when $\Sigma_{1}=I$, $\Sigma_{2}$ is a diagonal matrix and $z_{2}=0$.

In Subsection \ref{firstmainp1proofsubsection3}, we complete the proof of Statement (1) in Theorem \ref{firstmain} by reducing to the case when $\Sigma_{1}=I$, $\Sigma_{2}$ is a diagonal matrix, and $z_{2}=0$ using Lemma \ref{changeofbasislemma}.
\subsection{Log-concave function}\label{firstmainp1proofsubsection1}
\begin{lemma}\label{1dddlemma}
Suppose that $g:\mathbb{R}\to\mathbb{R}$ is a smooth function with $|g'(0)|\leq 1$ and $g''(0)\leq 1$, then
\[\left.\frac{d^{2}}{dt^{2}}\ln(1+e^{g(t)})\right|_{t=0}\leq 1.\]
\end{lemma}
\begin{proof}
The first derivative is given by
\[\frac{d}{dt}\ln(1+e^{g(t)})=\frac{e^{g(t)}g'(t)}{1+e^{g(t)}}.\]
The second derivative is given by
\begin{eqnarray*}
\frac{d^{2}}{dt^{2}}\ln(1+e^{g(t)})&=&
\frac{e^{g(t)}(g'(t)^{2}+g''(t))(1+e^{g(t)})-e^{2g(t)}g'(t)^{2}}{(1+e^{g(t)})^{2}}\\&=&
\frac{e^{g(t)}(g'(t)^{2}+g''(t))+e^{2g(t)}g''(t)}{(1+e^{g(t)})^{2}}.
\end{eqnarray*}
Since $|g'(0)|\leq 1$ and $g''(0)\leq 1$, it follows that
\[\left.\frac{d^{2}}{dt^{2}}\ln(1+e^{g(t)})\right|_{t=0}=
\frac{e^{g(0)}(g'(0)^{2}+g''(0))+e^{2g(0)}g''(0)}{(1+e^{g(0)})^{2}}\leq
\frac{2e^{g(0)}+e^{2g(0)}}{(1+e^{g(0)})^{2}}\leq 1.\]
\end{proof}
\begin{lemma}\label{Fconcave}
Suppose that $w\in\mathbb{R}^{n}$ with $\|w\|\leq\frac{1}{4}$ and $E\in\mathbb{R}^{n\times n}$ is a diagonal matrix with $\|E-I\|_{\mathrm{F}}\leq\frac{1}{4}$. Let $a,b>0$. Then the function $f:[-1,1]^{n}\to\mathbb{R}$ defined by
\[f(x)=ae^{-\|x-w\|^{2}}+be^{-\langle Ex,x\rangle}\]
is log-concave, i.e., the map $x\mapsto\ln(f(x))$ is concave.
\end{lemma}
\begin{proof}
Fix $x\in(-1,1)^{n}$ and $y\in\mathbb{R}^{n}$ with $\|y\|=1$. It suffices to show that the double derivative of $f(x+ty)$ with respect to $t$ evaluated at $t=0$ is less than or equal to $0$, i.e.,
\[\left.\frac{d^{2}}{dt^{2}}\ln f(x+ty)\right|_{t=0}\leq 0.\]
We have
\begin{eqnarray*}
\ln f(x+ty)&=&\ln\left(ae^{-\|x+ty-w\|^{2}}+be^{-\langle E(x+ty),(x+ty)\rangle}\right)\\&=&
\ln\left(ae^{-\|x+ty-w\|^{2}}\right)+\ln\left(1+\frac{b}{a}e^{\|x+ty-w\|^{2}-\langle E(x+ty),(x+ty)\rangle}\right)\\&=&
\ln a-\|x+ty-w\|^{2}+\ln\left(1+e^{g(t)}\right),
\end{eqnarray*}
where we define the function $g:\mathbb{R}\to\mathbb{R}$ by
\[g(t):=\ln\left(\frac{b}{a}\right)+\|x+ty-w\|^{2}-\langle E(x+ty),x+ty\rangle,\quad t\in\mathbb{R}.\]
If we can show that $|g'(0)|\leq 1$ and $g''(0)\leq 1$, then the proof is complete, since
\begin{eqnarray*}
\left.\frac{d^{2}}{dt^{2}}\ln f(x+ty)\right|_{t=0}&=&-2\|y\|^{2}+\left.\frac{d^{2}}{dt^{2}}\ln(1+e^{g(t)})\right|_{t=0}\\&=&
-2+\left.\frac{d^{2}}{dt^{2}}\ln(1+e^{g(t)})\right|_{t=0}\leq -1,
\end{eqnarray*}
where the second equality follows from the fact that $\|y\|=1$ and the last inequality follows from Lemma \ref{1dddlemma}.

Thus, it remains to show that $|g'(0)|\leq 1$ and $g''(0)\leq 1$. We have
\[g'(t)=2\langle x-w,y\rangle+2t\|y\|^{2}-2\langle Ex,y\rangle-2t\langle Ey,y\rangle,\]
for all $t\in\mathbb{R}$, and so
\[g'(0)=2\langle x-w-Ex,y\rangle\quad\text{and}\quad g''(0)=2(\|y\|^{2}-\langle Ey,y\rangle).\]
Since $\|y\|=1$ and $\|w\|\leq\frac{1}{4}$, this implies that
\begin{eqnarray*}
|g'(0)|\leq2\|x-w-Ex\|&\leq&
2\|x-Ex\|+\frac{1}{2}\\&=&
2\left(\sum_{i=1}^{n}(1-E_{i,i})^{2}\langle x,e_{i}\rangle^{2}\right)^{1/2}+\frac{1}{2}\\&\leq&
2\left(\sum_{i=1}^{n}(1-E_{i,i})^{2}\right)^{1/2}+\frac{1}{2}\leq 1,
\end{eqnarray*}
where $E_{i,i}$ is the $(i,i)$-entry of the diagonal matrix $E$, the fourth step follows from the fact that $x\in[-1,1]^{n}$ and the last step follows from the assumption $\|E-I\|_{\mathrm{F}}\leq\frac{1}{4}$. We also have
\[g''(0)=2\langle(I-E)y,y\rangle\leq2\|I-E\|_{\mathrm{op}}\leq\frac{1}{2}.\]
The result follows.
\end{proof}
\subsection{Diagonal covariance}\label{firstmainp1proofsubsection2}
In the next lemma, we prove some properties of some univariate functions that arise from the proof of Lemma \ref{firstmainp1diagonal} below. These functions enable us to select the suitable parameters for the probability measures we construct in order to fit the desired first and second moments. Indeed, $h_{1}(r)$ and $h_{2}(r)$ are the mean and second moment, respectively, of some random variable, whereas $h_{4}(r)$ is the second moment of another random variable. Property (3) in Lemma \ref{functionhlemma} below ensures that $h_{3}$ is invertible around $r=0$ and its inverse behaves well, whereas property (4) ensures the invertibility of $h_{4}$ around $\lambda=1$ and the well-behavedness of its inverse.
\begin{lemma}\label{functionhlemma}
Define the functions $h_{1},h_{2},h_{3},h_{4}$ on $\mathbb{R}$ by
\[h_{1}(r)=\frac{\int_{-1}^{1}te^{-(t-r)^{2}}\,dt}{\int_{-1}^{1}e^{-(t-r)^{2}}\,dt},\quad h_{2}(r)=\frac{\int_{-1}^{1}t^{2}e^{-(t-r)^{2}}\,dt}{\int_{-1}^{1}e^{-(t-r)^{2}}\,dt},\]
\[h_{3}(r)=\frac{h_{1}(r)}{\sqrt{h_{2}(r)-h_{1}(r)^{2}}},\quad h_{4}(\lambda)=\frac{\int_{-1}^{1}t^{2}e^{-\lambda t^{2}}\,dt}{\int_{-1}^{1}e^{-\lambda t^{2}}\,dt},\]
for $r,\lambda\in\mathbb{R}$. Then there exists an absolute constant $0<c_{1}<\frac{1}{2}$ such that
\begin{enumerate}[(1)]
\item $|h_{1}(r)|\leq|r|$ for all $r\in[-c_{1},c_{1}]$;
\item $|h_{2}(r)-h_{2}(0)|\leq|r|$ and $h_{2}(r)-h_{1}(r)^{2}>0$ for all $r\in[-c_{1},c_{1}]$;
\item $h_{3}$ is strictly increasing on $[-c_{1},c_{1}]$ and $|h_{3}(r)|\geq|r|$ for all $r\in[-c_{1},c_{1}]$;
\item $h_{4}$ is strictly decreasing on $[1-c_{1},1+c_{1}]$ and $|h_{4}(\lambda)-h_{4}(1)|\geq 0.06|\lambda-1|$ for all $\lambda\in[1-c_{1},1+c_{1}]$.
\end{enumerate}
\end{lemma}
\begin{proof}
We have $h_{1}(0)=0$,
\begin{eqnarray*}
h_{1}'(0)&=&\frac{(\int_{-1}^{1}2t^{2}e^{-t^{2}}\,dt)(\int_{-1}^{1}e^{-t^{2}}\,dt)-(\int_{-1}^{1}te^{-t^{2}}\,dt)(\int_{-1}^{1}2te^{-t^{2}}\,dt)}{(\int_{-1}^{1}e^{-t^{2}}\,dt)^{2}}\\&=&
\frac{\int_{-1}^{1}2t^{2}e^{-t^{2}}\,dt}{\int_{-1}^{1}e^{-t^{2}}\,dt}=0.507\ldots,
\end{eqnarray*}
$\displaystyle h_{2}(0)=\frac{\int_{-1}^{1}t^{2}e^{-t^{2}}\,dt}{\int_{-1}^{1}e^{-t^{2}}\,dt}>0$,
\begin{eqnarray*}
h_{2}'(0)&=&\frac{(\int_{-1}^{1}2t^{3}e^{-t^{2}}\,dt)(\int_{-1}^{1}e^{-t^{2}}\,dt)-(\int_{-1}^{1}t^{2}e^{-t^{2}}\,dt)(\int_{-1}^{1}2te^{-t^{2}}\,dt)}{(\int_{-1}^{1}e^{-t^{2}}\,dt)^{2}}=0,
\end{eqnarray*}
$h_{3}(0)=\frac{h_{1}(0)}{\sqrt{h_{2}(0)-h_{1}(0)^{2}}}=0$,
\[h_{3}'(0)=\frac{h_{1}'(0)}{\sqrt{h_{2}(0)-h_{1}(0)^{2}}}+0=2\sqrt{\frac{\int_{-1}^{1}t^{2}e^{-t^{2}}\,dt}{\int_{-1}^{1}e^{-t^{2}}\,dt}}=1.007\ldots,\]
\begin{eqnarray*}
h_{4}'(1)&=&\frac{(\int_{-1}^{1}-t^{4}e^{-t^{2}}\,dt)(\int_{-1}^{1}e^{-t^{2}}\,dt)-(\int_{-1}^{1}t^{2}e^{-t^{2}}\,dt)(\int_{-1}^{1}-t^{2}e^{-t^{2}}\,dt)}{(\int_{-1}^{1}e^{-t^{2}}\,dt)^{2}}\\&=&
-0.069\ldots.
\end{eqnarray*}
Thus, by continuity, there exists $c_{1}>0$ such that
\[|h_{1}'(r)|\leq 1,\quad |h_{2}'(r)|\leq 1,\quad h_{2}(r)-h_{1}(r)^{2}>0,\quad h_{3}'(r)\geq 1\quad\text{for all}\; r\in[-c_{1},c_{1}],\]
and
\[h_{4}'(\lambda)\leq-0.06\quad\text{for all}\;\lambda\in[1-c_{1},1+c_{1}].\]
Thus the result follows by the mean value theorem and the fact that $h_{1}(0)=h_{3}(0)=0$.
\end{proof}
\begin{lemma}\label{firstmainp1diagonal}
There exists an absolute constant $0<c_{3}<1$ such that the following holds. Suppose that $z\in\mathbb{R}^{n}$ with $\|z\|\leq c_{3}$ and $D\in\mathbb{R}^{n\times n}$ is diagonal with $\|D-I\|_{\mathrm{F}}\leq c_{3}$. Then there exist probability measures $\mu_{1}$ and $\mu_{2}$ on $\mathbb{R}^{n}$ such that
\begin{enumerate}[(1)]
\item $\mu_{1}$ has mean $z$ and covariance $I$;
\item $\mu_{2}$ has mean $0$ and covariance $D$;
\item $\alpha\mu_{1}+(1-\alpha)\mu_{2}$ is log-concave for all $\alpha\in[0,1]$.
\end{enumerate}
\end{lemma}
\begin{proof}
Let the functions $h_{1},h_{2},h_{3},h_{4}$ and the constant $0<c_{1}<\frac{1}{2}$ be as in Lemma \ref{functionhlemma}. Note that $h_{3}(0)=0$.

{\bf Step 1:} Define $c_{3}>0$.

Let $c_{2}>0$ be a small enough constant such that $[-c_{2},c_{2}]$ is contained in $h_{3}([-c_{1},c_{1}])$ and that $[h_{4}(1)-c_{2},h_{4}(1)+c_{2}]$ is contained in $h_{4}([1-c_{1},1+c_{1}])$. Take
\[c_{3}:=\min\left(\frac{c_{2}}{3},0.005\right).\]

The next two steps, Step 2 and Step 3, are essential for showing that the measures $\nu_{1},\nu_{2}$ constructed in Step 4 satisfy their desired properties. 

{\bf Step 2:} Show that there exist $r_{1},\ldots,r_{n}\in[-c_{1},c_{1}]$ and $\lambda_{1},\ldots,\lambda_{n}\in[1-c_{1},1+c_{1}]$ such that
\begin{equation}\label{firstmainp1diagonalproofeq1}
h_{3}(r_{i})=\langle z,e_{i}\rangle\quad\text{and}\quad h_{4}(\lambda_{i})=(h_{2}(r_{i})-h_{1}(r_{i})^{2})D_{i,i},
\end{equation}
for all $1\leq i\leq n$, where $D_{i,i}$ is the $(i,i)$-entry of the diagonal matrix $D$.

To prove this, we first construct $r_{1},\ldots,r_{n}$ as follows. Observe that since $\|z\|\leq c_{3}$, we have $|\langle z,e_{i}\rangle|\leq c_{3}\leq c_{2}$ and so by the definition of $c_{2}$ in Step 1, we have $\langle z,e_{i}\rangle\in h_{3}([-c_{1},c_{1}])$ for all $1\leq i\leq n$. Hence there exist $r_{1},\ldots,r_{n}\in[-c_{1},c_{1}]$ such that $h_{3}(r_{i})=\langle z,e_{i}\rangle$ for all $1\leq i\leq n$.

Next, we construct $\lambda_{1},\ldots,\lambda_{n}$ as follows. Observe that $h_{4}(1)=h_{2}(0)$ and so for all $1\leq i\leq n$, we have
\begin{align}\label{firstmainp1diagonalproofeq2}
&|(h_{2}(r_{i})-h_{1}(r_{i})^{2})D_{i,i}-h_{4}(1)|\\=&
|(h_{2}(r_{i})-h_{1}(r_{i})^{2})D_{i,i}-h_{2}(0)|\nonumber\\\leq&
|h_{2}(r_{i})D_{i,i}-h_{2}(0)|+|h_{1}(r_{i})^{2}D_{i,i}|\nonumber\\\leq&
|h_{2}(r_{i})|\cdot|D_{i,i}-1|+|h_{2}(r_{i})-h_{2}(0)|+|h_{1}(r_{i})^{2}D_{i,i}|\nonumber\\\leq&
|D_{i,i}-1|+|r_{i}|+|r_{i}|^{2}|D_{i,i}|\nonumber\\\leq&
|D_{i,i}-1|+2|r_{i}|\nonumber\\\leq&
|D_{i,i}-1|+2|h_{3}(r_{i})|\nonumber\\=&
|D_{i,i}-1|+2|\langle z,e_{i}\rangle|\leq 3c_{3}\leq c_{2},\nonumber
\end{align}
where the fourth step follows from the fact that $|h_{2}(r)|\leq 1$ for all $r\in\mathbb{R}$, Lemma \ref{functionhlemma}(1) and Lemma \ref{functionhlemma}(2), the fifth step follows from the fact that $|r_{i}|\leq c_{1}<\frac{1}{2}$ and the assumption $\|D-I\|_{\mathrm{F}}\leq c_{3}\leq 1$, the sixth step follows from Lemma \ref{functionhlemma}(3), the seventh step follows from the construction of $r_{i}$, and the eighth step follows from the assumptions $\|D-I\|_{\mathrm{F}}\leq c_{3}$ and $\|z\|\leq c_{3}$. Thus, by the definition of $c_{2}$ in Step 1, we have $(h_{2}(r_{i})-h_{1}(r_{i})^{2})D_{i,i}\in h_{4}([1-c_{1},1+c_{1}])$. So there exist $\lambda_{1},\ldots,\lambda_{n}\in[1-c_{1},1+c_{1}]$ such that $h_{4}(\lambda_{i})=(h_{2}(r_{i})-h_{1}(r_{i})^{2})D_{i,i}$ for all $1\leq i\leq n$. This completes Step 2.

{\bf Step 3:} Show that
\begin{equation}\label{firstmainp1diagonalproofeq3}
\left(\sum_{i=1}^{n}|r_{i}|^{2}\right)^{1/2}\leq\frac{1}{4}\quad\text{and}\quad\left(\sum_{i=1}^{n}|\lambda_{i}-1|^{2}\right)^{1/2}\leq\frac{1}{4}.
\end{equation}
To prove this, observe that by Lemma \ref{functionhlemma}(3) and (\ref{firstmainp1diagonalproofeq1}), we have
\begin{equation}\label{firstmainp1diagonalproofeq4}
\left(\sum_{i=1}^{n}|r_{i}|^{2}\right)^{1/2}\leq\left(\sum_{i=1}^{n}|h_{3}(r_{i})|^{2}\right)^{1/2}=\left(\sum_{i=1}^{n}|\langle z,e_{i}\rangle|^{2}\right)^{1/2}=\|z\|\leq c_{3}.
\end{equation}
This proves the first inequality in (\ref{firstmainp1diagonalproofeq3}). We now prove the other inequality in (\ref{firstmainp1diagonalproofeq3}). Recall from (\ref{firstmainp1diagonalproofeq2}) that
\begin{equation}\label{firstmainp1diagonalproofeq5}
|(h_{2}(r_{i})-h_{1}(r_{i})^{2})D_{i,i}-h_{4}(1)|\leq|D_{i,i}-1|+2|r_{i}|.
\end{equation}
We have
\begin{eqnarray*}
\left(\sum_{i=1}^{n}|\lambda_{i}-1|^{2}\right)^{1/2}&\leq&
\frac{1}{0.06}\left(\sum_{i=1}^{n}|h_{4}(\lambda_{i})-h_{4}(1)|^{2}\right)^{1/2}\\&=&
\frac{1}{0.06}\left(\sum_{i=1}^{n}|(h_{2}(r_{i})-h_{1}(r_{i})^{2})D_{i,i}-h_{4}(1)|^{2}\right)^{1/2}\\&\leq&
\frac{1}{0.06}\left(\sum_{i=1}^{n}(|D_{i,i}-1|+2|r_{i}|)^{2}\right)^{1/2}\\&\leq&
\frac{1}{0.06}\left(\sum_{i=1}^{n}|D_{i,i}-1|^{2}\right)^{1/2}+\frac{2}{0.06}\left(\sum_{i=1}^{n}|r_{i}|^{2}\right)^{1/2}\\&\leq&
\frac{3c_{3}}{0.06}\leq\frac{1}{4},
\end{eqnarray*}
where the first step follows from Lemma \ref{functionhlemma}(4), the second step follows from (\ref{firstmainp1diagonalproofeq1}), the third step follows from (\ref{firstmainp1diagonalproofeq5}), and the fifth step follows from the assumption that $\|D-I\|_{\mathrm{F}}\leq c_{3}$ and (\ref{firstmainp1diagonalproofeq4}). This completes Step 3.

{\bf Step 4:} Show that there exist probability measures $\nu_{1}$ and $\nu_{2}$ on $\mathbb{R}^{n}$ such that
\begin{enumerate}[(1)]
\item $\nu_{1}$ has mean $(h_{1}(r_{i}))_{1\leq i\leq n}$ and covariance $\mathrm{diag}((h_{2}(r_{i})-h_{1}(r_{i})^{2})_{1\leq i\leq n})$;
\item $\nu_{2}$ has mean $0$ and covariance $\mathrm{diag}((h_{4}(\lambda_{i}))_{1\leq i\leq n})$;
\item $\alpha\nu_{1}+(1-\alpha)\nu_{2}$ is log-concave for all $\alpha\in[0,1]$.
\end{enumerate}
To prove this, let $E=\mathrm{diag}(\lambda_{1},\ldots,\lambda_{n})$. Take $\nu_{1}$ and $\nu_{2}$ to be the probability measures supported on $[-1,1]^{n}$ with densities given by
\[f_{\nu_{1}}(x):=\frac{\exp(-\|x-(r_{1},\ldots,r_{n})\|^{2})}{\int_{[-1,1]^{n}}\exp(-\|y-(r_{1},\ldots,r_{n})\|^{2})\,dy}=\prod_{i=1}^{n}\frac{\exp(-(\langle x,e_{i}\rangle-r_{i})^{2})}{\int_{-1}^{1}\exp(-|t-r_{i}|^{2})\,dt},\]
and
\[f_{\nu_{2}}(x):=\frac{e^{-\langle Ex,x\rangle}}{\int_{[-1,1]^{d}}e^{-\langle Ey,y\rangle}\,dy}=\prod_{i=1}^{n}\frac{\exp(-\lambda_{i}\langle x,e_{i}\rangle^{2})}{\int_{-1}^{1}\exp(-\lambda_{i}t^{2})\,dt},\]
for $x\in[-1,1]^{n}$, respectively. By Step 3, we have $\|(r_{1},\ldots,r_{n})\|\leq\frac{1}{4}$ and $\|E-I\|_{\mathrm{F}}\leq\frac{1}{4}$. So by Lemma \ref{Fconcave}, the mixture $\alpha\nu_{1}+(1-\alpha)\nu_{2}$ is log-concave for all $\alpha\in[0,1]$.

Since $\nu_{1}$ is a product measure, if $(X_{1},\ldots,X_{d})$ is a random vector in $\mathbb{R}^{d}$ distributed according to $\nu_{1}$, then $X_{1},\ldots,X_{d}$ are independent random variables,
\[\mathbb{E}X_{i}=\frac{\int_{-1}^{1}te^{-(t-r_{i})^{2}}\,dt}{\int_{-1}^{1}e^{-(t-r_{i})^{2}}\,dt}=h_{1}(r_{i})\quad\text{and}\quad\mathbb{E}(X_{i}^{2})=\frac{\int_{-1}^{1}t^{2}e^{-(t-r_{i})^{2}}\,dt}{\int_{-1}^{1}e^{-(t-r_{i})^{2}}\,dt}=h_{2}(r_{i}).\]
Hence, $\nu_{1}$ has mean $(h_{1}(r_{i}))_{1\leq i\leq n}$ and covariance $\mathrm{diag}((h_{2}(r_{i})-h_{1}(r_{i})^{2})_{1\leq i\leq n})$. Similarly, we can show that $\nu_{2}$ has mean $0$ and covariance $\mathrm{diag}((h_{4}(\lambda_{i}))_{1\leq i\leq n})$.

{\bf Step 5:} Completing the proof.

Let $G=\mathrm{diag}\left((1/\sqrt{h_{2}(r_{i})-h_{1}(r_{i})^{2}}\,)_{1\leq i\leq n}\right)\in\mathbb{R}^{n\times n}$, where we note that the term inside the square root in the denominator is positive, i.e., $h_{2}(r_{i})-h_{1}(r_{i})^{2}>0$ by Lemma \ref{functionhlemma}(2). Take $\mu_{1}=G_{\#}\nu_{1}$ and $\mu_{2}=G_{\#}\nu_{2}$ to be the pushforward measures by $G$ of the measures $\nu_{1}$ and $\nu_{2}$ introduced in Step 4. By Step 4, we have
\begin{enumerate}[(1)]
\item The mean of $\mu_{1}$ is equal to
\[G\cdot\left[(h_{1}(r_{i}))_{1\leq i\leq n}\right]=(h_{1}(r_{i})/\sqrt{h_{2}(r_{i})-h_{1}(r_{i})^{2}})_{1\leq i\leq n}=(h_{3}(r_{i}))_{1\leq i\leq n}=z,\]
where the $\cdot$ denotes the usual matrix-vector multiplication, and the last equality follows from (\ref{firstmainp1diagonalproofeq1}). The covariance of $\mu_{1}$ is equal to
\[G\cdot\mathrm{diag}((h_{2}(r_{i})-h_{1}(r_{i})^{2})_{1\leq i\leq n})\cdot G=I,\]
where the $\cdot$ denotes the usual matrix multiplication.
\item The mean of $\mu_{2}$ is $0$ and the covariance of $\mu_{2}$ is equal to
\[G\cdot\mathrm{diag}((h_{4}(\lambda_{i}))_{1\leq i\leq n})\cdot G=\mathrm{diag}\left(\left(\frac{h_{4}(\lambda_{i})}{h_{2}(r_{i})-h_{1}(r_{i})^{2}}\right)_{1\leq i\leq n}\right)=D,\]
where the $\cdot$ denotes the usual matrix multiplication, and the last equality follows from (\ref{firstmainp1diagonalproofeq1}).
\item $\alpha\mu_{1}+(1-\alpha)\mu_{2}=G_{\#}(\alpha\nu_{1}+(1-\alpha)\nu_{2})$ is log-concave for all $\alpha\in[0,1]$.
\end{enumerate}
\end{proof}
\subsection{Completing the proof of the first main result part (1)}\label{firstmainp1proofsubsection3}
\begin{lemma}\label{changeofbasislemma}
Suppose that $\Sigma_{1},\Sigma_{2}\in\mathbb{R}^{n\times n}$ are positive semidefinite matrices such that $\Sigma_{1}+\Sigma_{2}$ is invertible and
\[\|(\Sigma_{1}+\Sigma_{2})^{-1/2}(\Sigma_{1}-\Sigma_{2})(\Sigma_{1}+\Sigma_{2})^{-1/2}\|_{\mathrm{F}}\leq\frac{c_{3}}{4}.\]
Then there exist an invertible matrix $S\in\mathbb{R}^{n\times n}$ and a diagonal matrix $D\in\mathbb{R}^{n\times n}$ such that $\Sigma_{1}=SS^{T}$, $\Sigma_{2}=SDS^{T}$ and $\|D-I\|_{\mathrm{F}}\leq c_{3}$. Moreover, $\|\Sigma_{1}^{-1/2}(\Sigma_{1}+\Sigma_{2})^{1/2}\|_{\mathrm{op}}\leq 2$.
\end{lemma}
\begin{proof}
{\bf Step 1:} Show that $\Sigma_{1}$ is invertible.

By assumption,
\[\left|\langle(\Sigma_{1}+\Sigma_{2})^{-1/2}(\Sigma_{1}-\Sigma_{2})(\Sigma_{1}+\Sigma_{2})^{-1/2}y,y\rangle\right|\leq\frac{1}{2}\|y\|^{2},\]
for all $y\in\mathbb{R}^{n}$, and so taking $y=(\Sigma_{1}+\Sigma_{2})^{1/2}x$ where $x\in\mathbb{R}^{n}$, we have
\begin{equation}\label{changeofbasislemmaproofeq1}
\left|\langle(\Sigma_{1}-\Sigma_{2})x,x\rangle\right|\leq\frac{1}{2}\|(\Sigma_{1}+\Sigma_{2})^{1/2}x\|^{2}=\frac{1}{2}\langle(\Sigma_{1}+\Sigma_{2})x,x\rangle,
\end{equation}
for all $x\in\mathbb{R}^{n}$. Thus, if $\Sigma_{1}x=0$, we have $\left|\langle\Sigma_{2}x,x\rangle\right|\leq\frac{1}{2}\langle\Sigma_{2}x,x\rangle$. Since $\Sigma_{2}$ is positive semidefinite, this implies that $\langle\Sigma_{2}x,x\rangle=0$, or equivalently, $\Sigma_{2}x=0$, and so $(\Sigma_{1}+\Sigma_{2})x=0$, but since $\Sigma_{1}+\Sigma_{2}$ is invertible, it follows that $x=0$. Therefore, $\Sigma_{1}$ is invertible.

{\bf Step 2:} Show that $\|\Sigma_{1}^{-1/2}(\Sigma_{1}+\Sigma_{2})^{1/2}\|_{\mathrm{op}}\leq 2$.

By (\ref{changeofbasislemmaproofeq1}), for every $x\in\mathbb{R}^{n}$, we have
\[\langle(\Sigma_{1}+\Sigma_{2})x,x\rangle=2\langle\Sigma_{1}x,x\rangle-\langle(\Sigma_{1}-\Sigma_{2})x,x\rangle\leq2\langle\Sigma_{1}x,x\rangle+\frac{1}{2}\langle(\Sigma_{1}+\Sigma_{2})x,x\rangle,\]
and so rearranging the terms, we obtain $\langle(\Sigma_{1}+\Sigma_{2})x,x\rangle\leq4\langle\Sigma_{1}x,x\rangle$. Since this holds for all $x\in\mathbb{R}^{n}$, replacing $x$ by $\Sigma_{1}^{-1/2}x$, we get
\[\langle\Sigma_{1}^{-1/2}(\Sigma_{1}+\Sigma_{2})\Sigma_{1}^{-1/2}x,x\rangle\leq4\|x\|^{2},\]
for all $x\in\mathbb{R}^{n}$, and so the spectral norm $\|\Sigma_{1}^{-1/2}(\Sigma_{1}+\Sigma_{2})\Sigma_{1}^{-1/2}\|_{\mathrm{op}}\leq 4$, or equivalently, $\|\Sigma_{1}^{-1/2}(\Sigma_{1}+\Sigma_{2})^{1/2}\|_{\mathrm{op}}\leq 2$.

{\bf Step 3:} Show that $\|\Sigma_{1}^{-1/2}\Sigma_{2}\Sigma_{1}^{-1/2}-I\|_{\mathrm{F}}\leq c_{3}$.
\begin{align*}
&\|\Sigma_{1}^{-1/2}\Sigma_{2}\Sigma_{1}^{-1/2}-I\|_{\mathrm{F}}\\=&
\|\Sigma_{1}^{-1/2}(\Sigma_{2}-\Sigma_{1})\Sigma_{1}^{-1/2}\|_{\mathrm{F}}\\\leq&
\|(\Sigma_{1}+\Sigma_{2})^{-1/2}(\Sigma_{2}-\Sigma_{1})(\Sigma_{1}+\Sigma_{2})^{-1/2}\|_{\mathrm{F}}\cdot\|\Sigma_{1}^{-1/2}(\Sigma_{1}+\Sigma_{2})^{1/2}\|_{\mathrm{op}}^{2}\\\leq&
\frac{c_{3}}{4}\cdot 2^{2}=c_{3},
\end{align*}
where the second last step follows from the assumption and Step 2.

{\bf Step 4:} Completing the proof.

Using the spectral decomposition, we can write $\Sigma_{1}^{-1/2}\Sigma_{2}\Sigma_{1}^{-1/2}=UDU^{T}$ for some orthogonal matrix $U\in\mathbb{R}^{n\times n}$ and diagonal matrix $D\in\mathbb{R}^{n\times n}$. Take $S=\Sigma_{1}^{1/2}U$. Since $\Sigma_{1}$ and $U$ are invertible, $S$ is also invertible. We also have $SS^{T}=\Sigma_{1}$,
\[SDS^{T}=\Sigma_{1}^{1/2}UDU^{T}\Sigma_{1}^{1/2}=\Sigma_{1}^{1/2}(\Sigma_{1}^{-1/2}\Sigma_{2}\Sigma_{1}^{-1/2})\Sigma_{1}^{1/2}=\Sigma_{2},\]
and
\[\|D-I\|_{\mathrm{F}}=\|U(D-I)U^{T}\|_{\mathrm{F}}=\|UDU^{T}-I\|_{\mathrm{F}}=\|\Sigma_{1}^{-1/2}\Sigma_{2}\Sigma_{1}^{-1/2}-I\|_{\mathrm{F}}\leq c_{3},\]
where the last step follows from Step 3.
\end{proof}
\begin{proof}[Proof of Theorem \ref{firstmain} Statement (1)]
We first prove it assuming that $\Sigma_{1}+\Sigma_{2}$ is invertible and $z_{2}=0$.

Take $c=c_{3}^{2}/4$ where $c_{3}$ is in Lemma \ref{firstmainp1diagonal}. Since $c\leq\frac{c_{3}}{4}$, by Lemma \ref{changeofbasislemma}, there exist an invertible matrix $S\in\mathbb{R}^{n\times n}$ and a diagonal matrix $D\in\mathbb{R}^{n\times n}$ such that $\Sigma_{1}=SS^{T}$, $\Sigma_{2}=SDS^{T}$ and $\|D-I\|_{\mathrm{F}}\leq c_{3}$.

Lemma \ref{changeofbasislemma} also gives that $\|\Sigma_{1}^{-1/2}(\Sigma_{1}+\Sigma_{2})^{1/2}\|_{\mathrm{op}}\leq 2$. Moreover, the assumption gives $\|(\Sigma_{1}+\Sigma_{2})^{-1/2}z_{1}\|^{2}=\|(\Sigma_{1}+\Sigma_{2})^{-1/2}(z_{1}-z_{2})\|^{2}\leq c$. Hence,
\[\|\Sigma_{1}^{-1/2}z_{1}\|^{2}\leq\|\Sigma_{1}^{-1/2}(\Sigma_{1}+\Sigma_{2})^{1/2}\|_{\mathrm{op}}^{2}\cdot\|(\Sigma_{1}+\Sigma_{2})^{-1/2}z_{1}\|^{2}\leq 2^{2}c=c_{3}^{2}.\]
So letting $z=S^{-1}z_{1}$, we have
\[\|z\|^{2}=\langle z,z\rangle=
\langle(SS^{T})^{-1}z_{1},z_{1}\rangle=
\langle\Sigma_{1}^{-1}z_{1},z_{1}\rangle=
\|\Sigma_{1}^{-1/2}z_{1}\|^{2}\leq c_{3}^{2},\]
so $\|z\|\leq c_{3}$. Since $\|z\|\leq c_{3}$ and $\|D-I\|_{\mathrm{F}}\leq c_{3}$, we can apply Lemma \ref{firstmainp1diagonal} to obtain probability measures $\mu_{1}$ and $\mu_{2}$ on $\mathbb{R}^{n}$ such that
\begin{enumerate}[(1)]
\item $\mu_{1}$ has mean $z$ and covariance $I$;
\item $\mu_{2}$ has mean $0$ and covariance $D$;
\item $\alpha\mu_{1}+(1-\alpha)\mu_{2}$ is log-concave for all $\alpha\in[0,1]$.
\end{enumerate}
Hence
\begin{enumerate}[(1)]
\item $S_{\#}\mu_{1}$ has mean $Sz=z_{1}$ and covariance $SS^{T}=\Sigma_{1}$;
\item $S_{\#}\mu_{2}$ has mean $0$ and covariance $SDS^{T}=\Sigma_{2}$;
\item $\alpha S_{\#}\mu_{1}+(1-\alpha)S_{\#}\mu_{2}=S_{\#}(\alpha\mu_{1}+(1-\alpha)\mu_{2})$ is log-concave for all $\alpha\in[0,1]$.
\end{enumerate}
This completes the proof of Statement (1) in Theorem \ref{firstmain} when $\Sigma_{1}+\Sigma_{2}$ is invertible and $z_{2}=0$.

Next we prove Statement (1) in Theorem \ref{firstmain} when $z_{2}=0$ but $\Sigma_{1}+\Sigma_{2}$ is not necessarily invertible. Since $z_{1}=z_{1}-z_{2}\in\mathrm{ran}(\Sigma_{1}+\Sigma_{2})$ and $\Sigma_{1}+\Sigma_{2}$ is an invertible operator on $\mathrm{ran}(\Sigma_{1}+\Sigma_{2})$, we can restrict the ambient space $\mathbb{R}^{n}$ to $\mathrm{ran}(\Sigma_{1}+\Sigma_{2})$ and using what we have just shown above, this completes the proof when $z_{2}=0$.

Finally, if $z_{2}\neq 0$, then we can apply Statement (1) in Theorem \ref{firstmain} with the same $\Sigma_{1},\Sigma_{2}$ but with $z_{1},z_{2}$ being replaced by $z_{1}-z_{2},0$, respectively, and then translate the resulting probability measures by $z_{2}$.
\end{proof}
\section{Non-isotropic thin shell}\label{nsthinshellsection}
In this section, we derive a concentration bound for the random variable $\langle AX,X\rangle$, where $A\in\mathbb{R}^{n\times n}$ and $X$ is a log-concave random vector in $\mathbb{R}^{n}$, using the recent result \cite{kls} on the KLS conjecture and the reverse H\"older inequality \cite{nsv}. This is the key ingredient used in the proof of Statement (2) in Theorem \ref{firstmain} and the proof of Theorem \ref{secondmain}.

A random vector $X$ in $\mathbb{R}^{n}$ is {\it isotropic} if $\mathbb{E}X=0$ and $\mathbb{E}(XX^{T})=I$. A random vector $X$ in $\mathbb{R}^{n}$ is {\it log-concave} if its distribution is log-concave.
\begin{lemma}\label{klsboundaxx}
If $X$ is an isotropic log-concave random vector in $\mathbb{R}^{n}$, then
\[\sqrt{\mathrm{Var}(\langle AX,X\rangle)}\leq C\sqrt{\ln n}\cdot\|A\|_{\mathrm{F}},\]
for all $A\in\mathbb{R}^{n\times n}$, where $C>0$ is an absolute constant.
\end{lemma}
\begin{proof}
By \cite[Theorem 1.2 and (1.4)]{kls}, we have
\[\mathrm{Var}(f(X))\leq C(\ln n)\mathbb{E}\|\nabla f(X)\|^{2},\]
for all locally-Lipschitz function $f:\mathbb{R}^{n}\to\mathbb{R}$, where $C>0$ is an absolute constant. Take $f(x)=\langle Ax,x\rangle$ for $x\in\mathbb{R}^{n}$. We have $\nabla f(x)=Ax+A^{T}x$ and so
\[\mathrm{Var}(\langle AX,X\rangle)\leq C(\ln n)\mathbb{E}\|AX+A^{T}X\|^{2}.\]
Since $X$ is isotropic, we have $\mathbb{E}\|AX\|^{2}=\|A\|_{\mathrm{F}}^{2}=\mathbb{E}\|A^{T}X\|^{2}$. The result follows.
\end{proof}
\begin{remark}
It is not known if Lemma \ref{klsboundaxx} holds without the $\sqrt{\ln n}$ factor. When $A=I$, the $\sqrt{\ln n}$ factor can be removed, and this is known as the thin shell conjecture, which has recently been solved \cite{thinshell}.
\end{remark}
\begin{lemma}\label{nsthinshell}
Suppose that $X$ is a log-concave random vector in $\mathbb{R}^{n}$ with $\mathbb{E}X=0$. Let $\nu$ be the distribution of $X$. Then
\begin{equation}\label{nsthinshelleq1}
\sqrt{\mathrm{Var}(\langle AX,X\rangle)}\leq C\sqrt{\ln n}\cdot\|B_{\nu}\|_{\mathrm{op}}\|A\|_{\mathrm{F}},
\end{equation}
and
\begin{equation}\label{nsthinshelleq2}
\nu(\{x\in\mathbb{R}^{n}:\,\left|\langle Ax,x\rangle-\mathrm{Tr}(AB_{\nu})\right|\geq t\sqrt{\ln n}\cdot\|B_{\nu}\|_{\mathrm{op}}\|A\|_{\mathrm{F}}\})\leq Ce^{-c\sqrt{t}},
\end{equation}
for all $t\geq 0$ and $A\in\mathbb{R}^{n\times n}$, where $C,c>0$ are absolute constants.
\end{lemma}
\begin{proof}
Without loss of generality, by restricting the ambient space $\mathbb{R}^{n}$ to $\mathrm{ran}(B_{\nu})$, if necessary, we may assume that the support of $\nu$ spans the entire $\mathbb{R}^{n}$ so that $B_{\nu}$ is invertible. The random vector $B_{\nu}^{-1/2}X$ is isotropic and log-concave. So by Lemma \ref{klsboundaxx},
\[\sqrt{\mathrm{Var}(\langle AB_{\nu}^{-1/2}X,B_{\nu}^{-1/2}X\rangle)}\leq C\sqrt{\ln n}\cdot\|A\|_{\mathrm{F}},\]
for all $A\in\mathbb{R}^{n\times n}$. Replacing $A$ by $B_{\nu}^{1/2}AB_{\nu}^{1/2}$, we obtain
\[\sqrt{\mathrm{Var}(\langle AX,X\rangle)}\leq C\sqrt{\ln n}\cdot\|B_{\nu}^{1/2}AB_{\nu}^{1/2}\|_{\mathrm{F}}\leq C\sqrt{\ln n}\cdot\|B_{\nu}\|_{\mathrm{op}}\|A\|_{\mathrm{F}},\]
for all $A\in\mathbb{R}^{n\times n}$. This proves (\ref{nsthinshelleq1}). To prove (\ref{nsthinshelleq2}), observe that $\mathbb{E}\langle AX,X\rangle=\mathrm{Tr}(A(\mathbb{E}XX^{T}))=\mathrm{Tr}(AB_{\nu})$, so we have
\[\sqrt{\mathbb{E}|\langle AX,X\rangle-\mathrm{Tr}(AB_{\nu})|^{2}}\leq C\sqrt{\ln n}\cdot\|B_{\nu}\|_{\mathrm{op}}\|A\|_{\mathrm{F}}.\]
Like in \cite[Corollary 1.3]{thinshell}, this bound on the second moment of $\langle AX,X\rangle-\mathrm{Tr}(AB_{\nu})$ implies (\ref{nsthinshelleq2}), since $\langle AX,X\rangle-\mathrm{Tr}(AB_{\nu})$ is a quadratic polynomial in $X$. This is due to the reverse H\"older inequality \cite[Distribution inequalities]{nsv} (or see \cite[Corollary 10 with $s=0$]{fradelizi}).
\end{proof}
\section{Proof of the first main result part (2)}\label{firstmainp2proofsection}
In this section, we prove Statement (2) in Theorem \ref{firstmain}. This actually consists of two separate parts: bounding $\|(\Sigma_{1}+\Sigma_{2})^{-1/2}(\Sigma_{1}-\Sigma_{2})(\Sigma_{1}+\Sigma_{2})^{-1/2}\|_{\mathrm{F}}$ and bounding $\|(\Sigma_{1}+\Sigma_{2})^{-1/2}(z_{1}-z_{2})\|$. The idea of the proof is as follows.

Since log-concavity is invariant under pushforward by any linear transformation, it suffices to prove the result when $\Sigma_{1}+\Sigma_{2}=I$. Thus, our goal is to bound $\|\Sigma_{1}-\Sigma_{2}\|_{\mathrm{F}}$ and $\|z_{1}-z_{2}\|^{2}$.

To bound $\|z_{1}-z_{2}\|^{2}$, we first obtain a bound in one dimension (see Lemma \ref{lcclosemean1d}), and then we show that the same bound holds for general dimension (see Lemma \ref{lcclosemean}), since we can project everything onto the one-dimensional subspace spanned by $z_{1}-z_{2}$.

To bound $\|\Sigma_{1}-\Sigma_{2}\|_{\mathrm{F}}$, observe that once we project the measures $\mu_{1},\mu_{2}$ onto the orthogonal complement of the span of $z_{1},z_{2}$, the means of those measures become $0$ and so their covariances coincide with their second moments. Thus, we can instead bound the difference between their second moments $\|B_{\mu_{1}}-B_{\mu_{2}}\|_{\mathrm{F}}$. Bounding $\|B_{\mu_{1}}-B_{\mu_{2}}\|_{\mathrm{F}}$ consists of two steps. First, we bound it by $\sup_{A\in\mathbb{R}^{n\times n},\,\|A\|_{\mathrm{F}}\leq 1}\sqrt{\mathrm{Var}(\langle AX,X\rangle)}$, where $X$ is a random vector in $\mathbb{R}^{n}$ distributed according to $\alpha\mu_{1}+(1-\alpha)\mu_{2}$ (see Lemma \ref{bktprop32}). Then we can apply the non-isotropic thin shell bound Lemma \ref{nsthinshell} to bound $\sup_{A\in\mathbb{R}^{n\times n},\,\|A\|_{\mathrm{F}}\leq 1}\sqrt{\mathrm{Var}(\langle AX,X\rangle)}$. Putting all these arguments together, we obtain Lemma \ref{lcclosecovariances} for bounding $\|\Sigma_{1}-\Sigma_{2}\|_{\mathrm{F}}$.

In Subsection \ref{firstmainp2proofsubsection1}, we prove Lemma \ref{lcclosemean}, which gives a bound for $\|z_{1}-z_{2}\|$ when $\|\Sigma_{1}\|_{\mathrm{op}}\leq 1$ and $\|\Sigma_{2}\|_{\mathrm{op}}\leq 1$.

In Subsection \ref{firstmainp2proofsubsection2}, we prove Lemma \ref{lcclosecovariances}, which gives a bound for $\|\Sigma_{1}-\Sigma_{2}\|_{\mathrm{F}}$.

In Subsection \ref{firstmainp2proofsubsection3}, we complete the proof of Statement (2) in Theorem \ref{firstmain} by first reducing to the case $\Sigma_{1}+\Sigma_{2}=I$ using a pushforward by a linear transformation and then applying Lemma \ref{lcclosemean} and Lemma \ref{lcclosecovariances}.
\subsection{Bounding the difference between the means}\label{firstmainp2proofsubsection1}
\begin{lemma}\label{lcsv1d}
Let $\alpha\in(0,1)$. There are no probability measures $\zeta_{1},\zeta_{2}$ on $\mathbb{R}$ such that
\begin{enumerate}[(1)]
\item $\zeta_{1}$ has mean $1$ and $\zeta_{2}$ has mean $0$
\item the variances of $\zeta_{1}$ and $\zeta_{2}$ are at most $(\frac{3}{16})^{2}\alpha(1-\alpha)$.
\item $\alpha\zeta_{1}+(1-\alpha)\zeta_{2}$ is log-concave.
\end{enumerate}
\end{lemma}
\begin{proof}
Suppose that such $\zeta_{1},\zeta_{2}$ exist. Let $\sigma=\frac{3}{16}\sqrt{\alpha(1-\alpha)}$. Then the variances of $\zeta_{1}$ and $\zeta_{2}$ are at most $\sigma^{2}$, so by Chebyshev's inequality,
\[\zeta_{1}((1-2\sigma,1+2\sigma))\geq\frac{3}{4}\quad\text{and}\quad\zeta_{2}((-2\sigma,2\sigma))\geq\frac{3}{4}.\]
Let $m=\lfloor\frac{1}{4\sigma}\rfloor\in\mathbb{N}$. Then $2\sigma\leq\frac{1}{2m}$. So
\[\zeta_{1}\left(\left(1-\frac{1}{2m},1+\frac{1}{2m}\right)\right)\geq\frac{3}{4}\quad\text{and}\quad\zeta_{2}\left(\left(-\frac{1}{2m},\frac{1}{2m}\right)\right)\geq\frac{3}{4},\]
and hence, for $b=0,1$, we have
\[(\alpha\zeta_{1}+(1-\alpha)\zeta_{2})\left(\left(b-\frac{1}{2m},b+\frac{1}{2m}\right)\right)\geq\frac{3}{4}\alpha^{b}(1-\alpha)^{1-b}.\]
By log-concavity of $\alpha\zeta_{1}+(1-\alpha)\zeta_{2}$, this holds for all $b\in[0,1]$. Since the intervals $(\frac{k}{m}-\frac{1}{2m},\frac{k}{m}+\frac{1}{2m})$, for $k=0,1,\ldots,m$, are disjoint, it follows that
\[\frac{3}{4}\sum_{k=0}^{m}\alpha^{\frac{k}{m}}(1-\alpha)^{1-\frac{k}{m}}\leq 1.\]
But by re-indexing $k$ to $m-k$ and then averaging with itself, we have
\begin{eqnarray*}
\sum_{k=0}^{m}\alpha^{\frac{k}{m}}(1-\alpha)^{1-\frac{k}{m}}&=&\sum_{k=0}^{m}\alpha^{\frac{m-k}{m}}(1-\alpha)^{1-\frac{m-k}{m}}\\&=&
\sum_{k=0}^{m}\frac{1}{2}\left(\alpha^{\frac{k}{m}}(1-\alpha)^{1-\frac{k}{m}}+\alpha^{\frac{m-k}{m}}(1-\alpha)^{1-\frac{m-k}{m}}\right)\\&\geq&
\sum_{k=0}^{m}\sqrt{\alpha(1-\alpha)}=(m+1)\sqrt{\alpha(1-\alpha)},
\end{eqnarray*}
where in the third step, we use the fact that $\frac{a+b}{2}\geq\sqrt{ab}$. Therefore,
\[\frac{3}{4}(m+1)\sqrt{\alpha(1-\alpha)}\leq 1.\]
This is impossible, since $m+1=\lfloor\frac{1}{4\sigma}\rfloor+1>\frac{1}{4\sigma}=\frac{4}{3\sqrt{\alpha(1-\alpha)}}$.
\end{proof}
\begin{lemma}\label{lcclosemean1d}
Let $\alpha\in(0,1)$. If $\zeta_{1},\zeta_{2}$ are probability measures on $\mathbb{R}$ with means $t_{1},t_{2}\in\mathbb{R}$, respectively, and variances at most $1$ and $\alpha\zeta_{1}+(1-\alpha)\zeta_{2}$ is log-concave, then
\[|t_{1}-t_{2}|\leq\frac{16}{3\sqrt{\alpha(1-\alpha)}}.\]
\end{lemma}
\begin{proof}
Suppose for contradiction that $|t_{1}-t_{2}|>\frac{16}{3\sqrt{\alpha(1-\alpha)}}$. Define $g:\mathbb{R}\to\mathbb{R}$ by $g(t)=\frac{t-t_{2}}{t_{1}-t_{2}}$ for $t\in\mathbb{R}$. Then
\begin{enumerate}[(1)]
\item $g_{\#}\zeta_{1}$ has mean $1$ and $g_{\#}\zeta_{2}$ has mean $0$;
\item the variances of $g_{\#}\zeta_{1}$ and $g_{\#}\zeta_{2}$ are at most $\frac{1}{|t_{1}-t_{2}|^{2}}<(\frac{3}{16})^{2}\alpha(1-\alpha)$.
\item $\alpha g_{\#}\zeta_{1}+(1-\alpha)g_{\#}\zeta_{2}=g_{\#}(\alpha\zeta_{1}+(1-\alpha)\zeta_{2})$ is log-concave.
\end{enumerate}
By Lemma \ref{lcsv1d}, this is impossible.
\end{proof}
\begin{lemma}\label{lcclosemean}
Let $\alpha\in(0,1)$. If $\nu_{1}$ and $\nu_{2}$ are probability measures on $\mathbb{R}^{n}$ with means $z_{1},z_{2}\in\mathbb{R}^{n}$, respectively, $\|\mathrm{Cov}(\nu_{1})\|_{\mathrm{op}}$, $\|\mathrm{Cov}(\nu_{2})\|_{\mathrm{op}}\leq 1$ and $\alpha\nu_{1}+(1-\alpha)\nu_{2}$ is log-concave, then
\[\|z_{1}-z_{2}\|\leq\frac{16}{3\sqrt{\alpha(1-\alpha)}}.\]
\end{lemma}
\begin{proof}
We may assume that $z_{1}\neq z_{2}$ (otherwise the conclusion of the result holds). Let $w=\frac{z_{1}-z_{2}}{\|z_{1}-z_{2}\|}$. Let $w_{\#}\nu_{1}$ and $w_{\#}\nu_{2}$ be the pushforward probability measures of $\nu_{1}$ and $\nu_{2}$, respectively, by the map $x\mapsto\langle x,w\rangle$. Note that these are measures on $\mathbb{R}$ and
\begin{enumerate}[(1)]
\item $w_{\#}\nu_{1}$ has mean $\langle z_{1},w\rangle$ and variance $\langle\mathrm{Cov}(\nu_{1})w,w\rangle\leq 1$, since $\|\mathrm{Cov}(\nu_{1})\|_{\mathrm{op}}\leq 1$;
\item $w_{\#}\nu_{2}$ has mean $\langle z_{2},w\rangle$ and variance $\langle\mathrm{Cov}(\nu_{2})w,w\rangle\leq 1$, since $\|\mathrm{Cov}(\nu_{2})\|_{\mathrm{op}}\leq 1$;
\item $\alpha w_{\#}\nu_{1}+(1-\alpha)w_{\#}\nu_{2}$ is log-concave.
\end{enumerate}
Thus, by Lemma \ref{lcclosemean1d},
\[\|z_{1}-z_{2}\|=|\langle z_{1}-z_{2},w\rangle|\leq\frac{16}{3\sqrt{\alpha(1-\alpha)}}\]
\end{proof}
\subsection{Bounding the difference between the covariances}\label{firstmainp2proofsubsection2}
\begin{lemma}[Proposition 3.2 in \cite{bkt}]\label{bktprop32}
Let $\alpha\in(0,1)$. Suppose that $\nu_{1},\nu_{2}\in\mathcal{P}_{2}(\mathbb{R}^{n})$ and $X$ is a random vector in $\mathbb{R}^{n}$ distributed according to the mixture $\alpha\nu_{1}+(1-\alpha)\nu_{2}$. Then
\[\|B_{\nu_{1}}-B_{\nu_{2}}\|_{\mathrm{F}}\leq\frac{1}{\alpha(1-\alpha)}\sup_{\|A\|_{\mathrm{F}}\leq 1}\sqrt{\mathrm{Var}(\langle AX,X\rangle)},\]
where the supremum is over all $A\in\mathbb{R}^{n\times n}$ with $\|A\|_{\mathrm{F}}\leq 1$.
\end{lemma}
\begin{proof}
For every $A\in\mathbb{R}^{n\times n}$ and every $b\in\mathbb{R}$, we have
\begin{eqnarray*}
\mathrm{Tr}(A(B_{\nu_{1}}-B_{\nu_{2}}))&=&
\int_{\mathbb{R}^{n}}\langle Ax,x\rangle\,d\nu_{1}(x)-\int_{\mathbb{R}^{n}}\langle Ax,x\rangle\,d\nu_{2}(x)\\&=&
\int_{\mathbb{R}^{n}}\langle Ax,x\rangle-b\,d\nu_{1}(x)-\int_{\mathbb{R}^{n}}\langle Ax,x\rangle-b\,d\nu_{2}(x)\\&\leq&
\int_{\mathbb{R}^{n}}|\langle Ax,x\rangle-b|\,d\nu_{1}(x)+\int_{\mathbb{R}^{n}}|\langle Ax,x\rangle-b|\,d\nu_{2}(x)\\&\leq&
\frac{1}{\alpha(1-\alpha)}\int_{\mathbb{R}^{n}}|\langle Ax,x\rangle-b|\,d(\alpha\nu_{1}+(1-\alpha)\nu_{2})(x)\\&=&
\frac{1}{\alpha(1-\alpha)}\mathbb{E}|\langle AX,X\rangle-b|\\&\leq&
\frac{1}{\alpha(1-\alpha)}\left(\mathbb{E}|\langle AX,X\rangle-b|^{2}\right)^{1/2}.
\end{eqnarray*}
The result follows by first taking $b=\mathbb{E}\langle AX,X\rangle$ and then taking supremum over all $A\in\mathbb{R}^{n\times n}$ with $\|A\|_{\mathrm{F}}\leq 1$.
\end{proof}
\begin{lemma}\label{gip}
If $M\in\mathbb{R}^{n\times n}$ and $P$ is an orthogonal projection on $\mathbb{R}^{n}$ with rank at most $2$, then
\[\|M\|_{\mathrm{F}}\leq\|(I-P)M(I-P)\|_{\mathrm{F}}+4\|M\|_{\mathrm{op}}.\]
\end{lemma}
\begin{proof}
\begin{eqnarray*}
\|M\|_{\mathrm{F}}&=&
\|(I-P)M(I-P)+PM(I-P)+MP\|_{\mathrm{F}}\\&\leq&
\|(I-P)M(I-P)\|_{\mathrm{F}}+\|PM\|_{\mathrm{F}}+\|MP\|_{\mathrm{F}}\\&\leq&
\|(I-P)M(I-P)\|_{\mathrm{F}}+4\|M\|_{\mathrm{op}}.
\end{eqnarray*}
\end{proof}
\begin{lemma}\label{lcclosecovariances}
Let $\alpha\in(0,1)$. Suppose that $\nu_{1}$ and $\nu_{2}$ are probability measures on $\mathbb{R}^{n}$ such that $\alpha\nu_{1}+(1-\alpha)\nu_{2}$ is log-concave. Then
\[\|\mathrm{Cov}(\nu_{1})-\mathrm{Cov}(\nu_{2})\|_{\mathrm{F}}\leq\frac{C\sqrt{\ln n}}{\alpha(1-\alpha)}\left(\|\mathrm{Cov}(\nu_{1})\|_{\mathrm{op}}+\|\mathrm{Cov}(\nu_{2})\|_{\mathrm{op}}\right),\]
where $C>0$ is an absolute constant.
\end{lemma}
\begin{proof}
Let $z_{1}$ and $z_{2}$ be the means of $\nu_{1}$ and $\nu_{2}$, respectively. Let $P$ be orthogonal projection from $\mathbb{R}^{n}$ onto the span of $z_{1},z_{2}$. Then $I-P$ is an operator on $\mathbb{R}^{n}$. Consider the pushforward measures $\widetilde{\nu}_{1}:=(I-P)_{\#}\nu_{1}$ and $\widetilde{\nu}_{2}:=(I-P)_{\#}\nu_{2}$ on $\mathbb{R}^{n}$. Observe that
\begin{enumerate}[(1)]
\item $\widetilde{\nu}_{1}$ has mean $(I-P)z_{1}=0$ and covariance $(I-P)\mathrm{Cov}(\nu_{1})(I-P)$
\item $\widetilde{\nu}_{2}$ has mean $(I-P)z_{2}=0$ and covariance $(I-P)\mathrm{Cov}(\nu_{2})(I-P)$
\item $\alpha\widetilde{\nu}_{1}+(1-\alpha)\widetilde{\nu}_{2}$ is log-concave.
\end{enumerate}
Let $\widetilde{\nu}:=\alpha\widetilde{\nu}_{1}+(1-\alpha)\widetilde{\nu}_{2}$ and let $X$ be a random vector in $\mathbb{R}^{n}$ distributed according to $\widetilde{\nu}$. By Lemma \ref{bktprop32} and Lemma \ref{nsthinshell}, we have
\[\|B_{\widetilde{\nu}_{1}}-B_{\widetilde{\nu}_{2}}\|_{\mathrm{F}}\leq\frac{1}{\alpha(1-\alpha)}\sup_{\|A\|_{\mathrm{F}}\leq 1}\sqrt{\mathrm{Var}(\langle AX,X\rangle)}\leq\frac{C\sqrt{\ln n}}{\alpha(1-\alpha)}\cdot\|B_{\widetilde{\nu}}\|_{\mathrm{op}}.\]
But
\begin{eqnarray*}
\|B_{\widetilde{\nu}}\|_{\mathrm{op}}&=&\|\alpha B_{\widetilde{\nu}_{1}}+(1-\alpha)B_{\widetilde{\nu}_{2}}\|_{\mathrm{op}}\\&\leq&
\|B_{\widetilde{\nu}_{1}}+B_{\widetilde{\nu}_{2}}\|_{\mathrm{op}}\\&=&
\|(I-P)(\mathrm{Cov}(\nu_{1})+\mathrm{Cov}(\nu_{2}))(I-P)\|_{\mathrm{op}}\\&\leq&
\|\mathrm{Cov}(\nu_{1})+\mathrm{Cov}(\nu_{2})\|_{\mathrm{op}}.
\end{eqnarray*}
Therefore,
\[\|B_{\widetilde{\nu}_{1}}-B_{\widetilde{\nu}_{2}}\|_{\mathrm{F}}\leq\frac{C\sqrt{\ln n}}{\alpha(1-\alpha)}\cdot\|\mathrm{Cov}(\nu_{1})+\mathrm{Cov}(\nu_{2})\|_{\mathrm{op}}.\]
Finally, the result follows by observing that, using Lemma \ref{gip}, we have
\begin{align*}
&\|\mathrm{Cov}(\nu_{1})-\mathrm{Cov}(\nu_{2})\|_{\mathrm{F}}\\\leq&
\|(I-P)(\mathrm{Cov}(\nu_{1})-\mathrm{Cov}(\nu_{2}))(I-P)\|_{\mathrm{F}}+4\|\mathrm{Cov}(\nu_{1})-\mathrm{Cov}(\nu_{2})\|_{\mathrm{op}}\\=&
\|B_{\widetilde{\nu}_{1}}-B_{\widetilde{\nu}_{2}}\|_{\mathrm{F}}+
4\|\mathrm{Cov}(\nu_{1})-\mathrm{Cov}(\nu_{2})\|_{\mathrm{op}}.
\end{align*}
\end{proof}
\subsection{Completing the proof of the first main result part (2)}\label{firstmainp2proofsubsection3}
\begin{proof}[Proof of Theorem \ref{firstmain} Statement (2)]
We first prove it assuming that $\Sigma_{1}+\Sigma_{2}$ is invertible.

Let $\nu_{1}=(\Sigma_{1}+\Sigma_{2})_{\#}^{-1/2}\mu_{1}$ and $\nu_{2}=(\Sigma_{1}+\Sigma_{2})_{\#}^{-1/2}\mu_{2}$ be the pushforward measures of $\mu_{1}$ and $\mu_{2}$, respectively, by the operator $(\Sigma_{1}+\Sigma_{2})^{-1/2}$. We have
\begin{enumerate}[(1)]
\item $\nu_{1}$ has mean $(\Sigma_{1}+\Sigma_{2})^{-1/2}z_{1}$ and covariance $(\Sigma_{1}+\Sigma_{2})^{-1/2}\Sigma_{1}(\Sigma_{1}+\Sigma_{2})^{-1/2}$
\item $\nu_{2}$ has mean $(\Sigma_{1}+\Sigma_{2})^{-1/2}z_{2}$ and covariance $(\Sigma_{1}+\Sigma_{2})^{-1/2}\Sigma_{2}(\Sigma_{1}+\Sigma_{2})^{-1/2}$
\item $\alpha\nu_{1}+(1-\alpha)\nu_{2}$ is log-concave.
\end{enumerate}
Thus, $\mathrm{Cov}(\nu_{1})+\mathrm{Cov}(\nu_{2})=I$, so $\|\mathrm{Cov}(\nu_{1})\|_{\mathrm{op}}\leq 1$ and $\|\mathrm{Cov}(\nu_{2})\|_{\mathrm{op}}\leq 1$. So by Lemma \ref{lcclosemean},
\[\|(\Sigma_{1}+\Sigma_{2})^{-1/2}(z_{1}-z_{2})\|\leq\frac{16}{3\sqrt{\alpha(1-\alpha)}},\]
and by Lemma \ref{lcclosecovariances},
\[\|(\Sigma_{1}+\Sigma_{2})^{-1/2}(\Sigma_{1}-\Sigma_{2})(\Sigma_{1}+\Sigma_{2})^{-1/2}\|_{\mathrm{F}}\leq\frac{C\sqrt{\ln n}}{\alpha(1-\alpha)}.\]
This completes the proof of Statement (2) in Theorem \ref{firstmain} when $\Sigma_{1}+\Sigma_{2}$ is invertible.

Next, we prove Statement (2) in Theorem \ref{firstmain} when $z_{2}=0$ but without assuming that $\Sigma_{1}+\Sigma_{2}$ is invertible. We need to show that $z_{1}\in\mathrm{ran}(\Sigma_{1}+\Sigma_{2})$. Let $Q$ be the orthogonal projection from $\mathbb{R}^{n}$ onto the orthogonal complement $\mathrm{ran}(\Sigma_{1}+\Sigma_{2})^{\perp}$, which is equal to $\mathrm{ran}(\Sigma_{1})^{\perp}\cap\mathrm{ran}(\Sigma_{2})^{\perp}$ (see Appendix \ref{inversesection}). Since $\alpha\mu_{1}+(1-\alpha)\mu_{2}$ is log-concave, $\alpha Q_{\#}\mu_{1}+(1-\alpha)Q_{\#}\mu_{2}$ is also log-concave. Observe that $Q_{\#}\mu_{1}$ has mean $Qz_{1}$ and covariance $Q\Sigma_{1}Q=0$, whereas $Q_{\#}\mu_{2}$ has mean $Qz_{2}=0$ and covariance $Q\Sigma_{2}Q=0$. Thus, we have $Q_{\#}\mu_{1}=\delta_{Qz_{1}}$ and $Q_{\#}\mu_{2}=\delta_{0}$, where $\delta_{x}$ denotes the probability measure with an atom at $x$ of mass $1$. But the only way $\alpha Q_{\#}\mu_{1}+(1-\alpha)Q_{\#}\mu_{2}=\alpha\delta_{Qz_{1}}+(1-\alpha)\delta_{0}$ can be log-concave is that $Qz_{1}=0$. Thus, $z_{1}\in\mathrm{ran}(\Sigma_{1}+\Sigma_{2})$. Now, we have $Q_{\#}\mu_{1}=Q_{\#}\mu_{2}=\delta_{0}$ and so the measures $\mu_{1}$ and $\mu_{2}$ are both supported on $\mathrm{ran}(\Sigma_{1}+\Sigma_{2})$. Moreover, $\Sigma_{1}+\Sigma_{2}$ is an invertible operator on $\mathrm{ran}(\Sigma_{1}+\Sigma_{2})$. Therefore, by restricting the ambient space $\mathbb{R}^{n}$ to $\mathrm{ran}(\Sigma_{1}+\Sigma_{2})$ and using what we have shown above, this completes the proof when $z_{2}=0$.

Finally, if $z_{2}\neq 0$, then we can translate the probability measures $\mu_{1}$ and $\mu_{2}$ by the vector $-z_{2}$ so that the translated $\mu_{2}$ has mean $0$ and apply what we have shown above.
\end{proof}
\section{Proof of the second main result}\label{secondmainproofsection}
In this section, we prove Theorem \ref{secondmain}. This result actually has 3 separate parts. We need to show that if at least one of the conditions holds
\begin{enumerate}[(i)]
\item $z_{1}-z_{2}\notin\mathrm{ran}(\Sigma_{1}+\Sigma_{2})$, or
\item $z_{1}-z_{2}\in\mathrm{ran}(\Sigma_{1}+\Sigma_{2})$ and $\|(\Sigma_{1}+\Sigma_{2})^{-1/2}(z_{1}-z_{2})\|^{2}$ is large enough, or
\item $\|(\Sigma_{1}+\Sigma_{2})^{-1/2}(\Sigma_{1}-\Sigma_{2})(\Sigma_{1}+\Sigma_{2})^{-1/2}\|$ is large enough,
\end{enumerate}
then there is a quadratic classifier set $L$ separating $\mu_{1}$ and $\mu_{2}$ and $L$ depends only on $z_{1},z_{2},\Sigma_{1},\Sigma_{2}$. The idea of the proof is as follows.

If condition (i) holds, the proof is straightforward, since $\mu_{1}$ and $\mu_{2}$ are supported on parallel subspaces $z_{1}+\mathrm{ran}(\Sigma_{1}+\Sigma_{2})$ and $z_{2}+\mathrm{ran}(\Sigma_{1}+\Sigma_{2})$ and so we can use a hyperplane to separate $\mu_{1}$ and $\mu_{2}$.

For the proofs when condition (ii) holds or when condition (iii) holds, by using a pushforward by a linear transformation, it suffices to prove it when $\Sigma_{1}+\Sigma_{2}=I$. Now, assume $\Sigma_{1}+\Sigma_{2}=I$ and we have $\|\Sigma_{1}\|_{\mathrm{op}}\leq 1$ and $\|\Sigma_{2}\|_{\mathrm{op}}\leq 1$.

If condition (ii) holds, then the means $z_{1}$ and $z_{2}$ of $\mu_{1}$ and $\mu_{2}$, respectively, are far apart, and we can separate $\mu_{1}$ and $\mu_{2}$ by a hyperplane (see Lemma \ref{meanseparation}). Indeed, if $x\in\mathbb{R}^{n}$ is sampled from $\mu_{1}$, then with high probability, we have
\[\left\langle x,\frac{z_{1}-z_{2}}{\|z_{1}-z_{2}\|}\right\rangle=\left\langle z_{1},\frac{z_{1}-z_{2}}{\|z_{1}-z_{2}\|}\right\rangle+O(1),\]
whereas if $x\in\mathbb{R}^{n}$ is sampled from $\mu_{2}$, then with high probability, we have
\[\left\langle x,\frac{z_{1}-z_{2}}{\|z_{1}-z_{2}\|}\right\rangle=\left\langle z_{2},\frac{z_{1}-z_{2}}{\|z_{1}-z_{2}\|}\right\rangle+O(1).\]
But when $z_{1}$ and $z_{2}$ are far apart, the numbers $\left\langle z_{1},\frac{z_{1}-z_{2}}{\|z_{1}-z_{2}\|}\right\rangle$ and $\left\langle z_{2},\frac{z_{1}-z_{2}}{\|z_{1}-z_{2}\|}\right\rangle$ will be far apart.

If condition (iii) holds, then the covariances $\Sigma_{1}$ and $\Sigma_{2}$ are far apart, and we can separate $\mu_{1}$ and $\mu_{2}$ using a quadratic set (see Lemma \ref{covseparation}). Indeed, far apart covariances imply far apart second moments after we project onto the orthogonal complements of the span of $z_{1},z_{2}$. But two probability measures having far apart second moments implies that after we lift everything by the map $x\mapsto xx^{T}\in\mathbb{R}^{n\times n}$, the means of the lifted measures are far apart and we can apply the above argument where we show that the result holds when condition (ii). The key ingredient that enables this to work for the lifted measures is the non-isotropic thin-shell bound Lemma \ref{nsthinshell}.

In Subsection \ref{secondmainproofsubsection1}, we prove Lemma \ref{meanseparation}, which says that we can separate $\mu_{1}$ and $\mu_{2}$ by a hyperplane when $z_{1}$ and $z_{2}$ are far apart, $\|\mathrm{Cov}(\mu_{1})\|_{\mathrm{op}}\leq 1$ and $\|\mathrm{Cov}(\mu_{2})\|_{\mathrm{op}}\leq 1$.

In Subsection \ref{secondmainproofsubsection2}, we first obtain Lemma \ref{smseparation}, which proves separation of $\mu_{1}$ and $\mu_{2}$ when the second moments are far apart. We then obtain Lemma \ref{covseparation}, which proves separation when the covariances are far apart.

In Subsection \ref{secondmainproofsubsection3}, we complete the proof of Theorem \ref{secondmain} by first reducing to the case $\Sigma_{1}+\Sigma_{2}=I$ using a pushforward by a linear transformation and then applying Lemma \ref{meanseparation} and Lemma \ref{covseparation}.
\subsection{Separation when the means are far apart}\label{secondmainproofsubsection1}
\begin{lemma}[\cite{agm}, Remark 3.5.12]\label{lcsubexp}
If $X$ is a log-concave random vector in $\mathbb{R}^{n}$, then
\[\mathbb{P}(|\langle X,w\rangle|\geq t\cdot\mathbb{E}|\langle X,w\rangle|)\leq 2e^{-ct},\]
for all $t\geq 0$ and $w\in\mathbb{R}^{n}$, where $c>0$ is an absolute constant.
\end{lemma}
\begin{lemma}\label{lcsubexp2}
If $\nu$ is a log-concave probability measure on $\mathbb{R}^{n}$ with mean $z\in\mathbb{R}^{n}$ and $\|\mathrm{Cov}(\nu)\|_{\mathrm{op}}\leq 1$, then
\[\nu(\{x\in\mathbb{R}^{n}:\,|\langle x-z,w\rangle|\geq t\|w\|\})\leq 2e^{-ct},\]
for all $t\geq 0$ and $w\in\mathbb{R}^{n}$, where $c>0$ is an absolute constant.
\end{lemma}
\begin{proof}
Suppose that $X$ is a random vector in $\mathbb{R}^{n}$ distributed according to $\nu$. Then $\mathbb{E}X=z$ and
\[\mathbb{E}|\langle X-z,w\rangle|\leq(\mathbb{E}|\langle X-z,w\rangle|^{2})^{1/2}\leq(\|\mathrm{Cov}(\nu)\|_{\mathrm{op}}\|w\|^{2})^{1/2}\leq\|w\|,\]
for all $w\in\mathbb{R}^{n}$. Moreover, since $X$ is log-concave, $X-z$ is also log-concave, so by Lemma \ref{lcsubexp},
\begin{align*}
&\mathbb{P}(|\langle X-z,w\rangle|\geq t\|w\|)\\\leq&
\mathbb{P}(|\langle X-z,w\rangle|\geq t\cdot\mathbb{E}|\langle X-z,w\rangle|)\leq 2e^{-ct},
\end{align*}
for all $t\geq 0$ and $w\in\mathbb{R}^{n}$.
\end{proof}
\begin{lemma}\label{meanseparation}
Let $0<\delta<\frac{1}{2}$. Let $\nu_{1},\nu_{2}$ be log-concave probability measures on $\mathbb{R}^{n}$ with means $z_{1},z_{2}$, respectively. If $\|\mathrm{Cov}(\nu_{1})\|_{\mathrm{op}}\leq 1$, $\|\mathrm{Cov}(\nu_{2})\|_{\mathrm{op}}\leq 1$ and
\[\|z_{1}-z_{2}\|\geq C_{1}\ln\frac{1}{\delta},\]
where $C_{1}>0$ is a large enough absolute constant, then
\[\nu_{1}(\{x\in\mathbb{R}^{n}:\,\langle x,z_{1}-z_{2}\rangle>s\})\geq 1-\delta,\]
and
\[\nu_{2}(\{x\in\mathbb{R}^{n}:\,\langle x,z_{1}-z_{2}\rangle<s\})\geq 1-\delta,\]
where $s=\frac{1}{2}\langle z_{1}+z_{2},z_{1}-z_{2}\rangle$.
\end{lemma}
\begin{proof}
Let $w=z_{1}-z_{2}$. Then $s=\frac{1}{2}\langle z_{1}+z_{2},w\rangle$, so we have
\[s-\langle z_{1},w\rangle=\frac{1}{2}\langle z_{2}-z_{1},w\rangle=-\frac{1}{2}\|w\|^{2},\]
and similarly,
\[s-\langle z_{2},w\rangle=\frac{1}{2}\langle z_{1}-z_{2},w\rangle=\frac{1}{2}\|w\|^{2}.\]
Thus, by Lemma \ref{lcsubexp2} with $t=\frac{1}{2}\|w\|$,
\begin{align*}
&\nu_{1}(\{x\in\mathbb{R}^{n}:\,\langle x,z_{1}-z_{2}\rangle\leq s\})\\=&
\nu_{1}(\{x\in\mathbb{R}^{n}:\,\langle x,w\rangle\leq s\})\\=&
\nu_{1}(\{x\in\mathbb{R}^{n}:\,\langle x-z_{1},w\rangle\leq s-\langle z_{1},w\rangle\})\\=&
\nu_{1}(\{x\in\mathbb{R}^{n}:\,\langle x-z_{1},w\rangle\leq-\frac{1}{2}\|w\|^{2}\})\\\leq&
2e^{-\frac{c}{2}\|w\|}\\\leq&
\delta,
\end{align*}
where the last step follows from the assumption $\|w\|=\|z_{1}-z_{2}\|\geq C_{1}\ln\frac{1}{\delta}$ and we choose the absolute constant $C_{1}>0$ to be large enough. Similarly, we also have
\begin{align*}
&\nu_{2}(\{x\in\mathbb{R}^{n}:\,\langle x,z_{1}-z_{2}\rangle\geq s\})\\=&
\nu_{2}(\{x\in\mathbb{R}^{n}:\,\langle x,w\rangle\geq s\})\\=&
\nu_{2}(\{x\in\mathbb{R}^{n}:\,\langle x-z_{2},w\rangle\geq s-\langle z_{2},w\rangle\})\\=&
\nu_{2}(\{x\in\mathbb{R}^{n}:\,\langle x-z_{2},w\rangle\geq\frac{1}{2}\|w\|^{2}\})\\\leq&
2e^{-\frac{c}{2}\|w\|}\\\leq&\delta.
\end{align*}
\end{proof}
\subsection{Separation when the covariances are far apart}\label{secondmainproofsubsection2}
\begin{lemma}\label{smseparation}
Let $0<\delta<\frac{1}{2}$. If $\nu_{1},\nu_{2}$ are log-concave probability measures on $\mathbb{R}^{n}$ such that $\nu_{1},\nu_{2}$ have mean $0$, $\|B_{\nu_{1}}\|_{\mathrm{op}}\leq 1$, $\|B_{\nu_{2}}\|_{\mathrm{op}}\leq 1$ and
\[\|B_{\nu_{1}}-B_{\nu_{2}}\|_{\mathrm{F}}\geq C_{2}\left(\ln\frac{1}{\delta}\right)^{2}\sqrt{\ln n},\]
where $C_{2}>0$ is a large enough absolute constant, then
\[\nu_{1}\left(\left\{x\in\mathbb{R}^{n}:\,\langle(B_{\nu_{1}}-B_{\nu_{2}})x,x\rangle>t\right\}\right)\geq 1-\delta,\]
and
\[\nu_{2}\left(\left\{x\in\mathbb{R}^{n}:\,\langle(B_{\nu_{1}}-B_{\nu_{2}})x,x\rangle<t\right\}\right)\geq 1-\delta,\]
where $t=\frac{1}{2}\mathrm{Tr}((B_{\nu_{1}}-B_{\nu_{2}})(B_{\nu_{1}}+B_{\nu_{2}}))$.
\end{lemma}
\begin{proof}
Let $A=B_{\nu_{1}}-B_{\nu_{2}}$. Then $t=\frac{1}{2}\mathrm{Tr}(A(B_{\nu_{1}}+B_{\nu_{2}}))$, so we have
\[t-\mathrm{Tr}(AB_{\nu_{1}})=\frac{1}{2}\mathrm{Tr}(A(B_{\nu_{2}}-B_{\nu_{1}}))=-\frac{1}{2}\mathrm{Tr}(A^{2})=-\frac{1}{2}\|A\|_{\mathrm{F}}^{2},\]
and similarly,
\[t-\mathrm{Tr}(AB_{\nu_{2}})=\frac{1}{2}\mathrm{Tr}(A(B_{\nu_{1}}-B_{\nu_{2}}))=\frac{1}{2}\mathrm{Tr}(A^{2})=\frac{1}{2}\|A\|_{\mathrm{F}}^{2}.\]
Thus, by Lemma \ref{nsthinshell},
\begin{align*}
&\nu_{1}(\{x\in\mathbb{R}^{n}:\,\langle Ax,x\rangle\leq t\})\\=&
\nu_{1}(\{x\in\mathbb{R}^{n}:\,\langle Ax,x\rangle-\mathrm{Tr}(AB_{\nu_{1}})\leq-\frac{1}{2}\|A\|_{\mathrm{F}}^{2}\})\\\leq&
C\exp\left(-c\sqrt{\frac{\|A\|_{\mathrm{F}}}{2\sqrt{\ln n}\cdot\|B_{\nu_{1}}\|_{\mathrm{op}}}}\,\right)\\\leq&
C\exp\left(-c\sqrt{\frac{\|A\|_{\mathrm{F}}}{2\sqrt{\ln n}}}\,\right)\\\leq&
\delta,
\end{align*}
where the last step follows from the assumption $\|A\|_{\mathrm{F}}=\|B_{\nu_{1}}-B_{\nu_{2}}\|_{\mathrm{F}}\geq C_{2}(\ln\frac{1}{\delta})^{2}\sqrt{\ln n}$ and we choose the absolute constant $C_{2}>0$ to be large enough. Similarly, we also have
\begin{align*}
&\nu_{2}(\{x\in\mathbb{R}^{n}:\,\langle Ax,x\rangle\geq t\})\\=&
\nu_{2}(\{x\in\mathbb{R}^{n}:\,\langle Ax,x\rangle-\mathrm{Tr}(AB_{\nu_{2}})\geq\frac{1}{2}\|A\|_{\mathrm{F}}^{2}\})\\\leq&
C\exp\left(-c\sqrt{\frac{\|A\|_{\mathrm{F}}}{2\sqrt{\ln n}\cdot\|B_{\nu_{2}}\|_{\mathrm{op}}}}\right)\\\leq&
C\exp\left(-c\sqrt{\frac{\|A\|_{\mathrm{F}}}{2\sqrt{\ln n}}}\right)\\\leq&
\delta.
\end{align*}
\end{proof}
\begin{lemma}\label{covseparation}
Let $0<\delta<\frac{1}{2}$. If $\nu_{1},\nu_{2}$ are log-concave probability measures on $\mathbb{R}^{n}$ such that $\|\mathrm{Cov}(\nu_{1})\|_{\mathrm{op}}\leq 1$, $\|\mathrm{Cov}(\nu_{2})\|_{\mathrm{op}}\leq 1$ and
\[\|\mathrm{Cov}(\nu_{1})-\mathrm{Cov}(\nu_{2})\|_{\mathrm{F}}\geq C_{2}\left(\ln\frac{1}{\delta}\right)^{2}\sqrt{\ln n}+4,\]
where $C_{2}>0$ is a large enough absolute constant, then there is a quadratic set $L\subset\mathbb{R}^{n}$ that depends only on the means and covariances of $\nu_{1},\nu_{2}$ such that
\[\nu_{1}(L)\geq 1-\delta\quad\text{and}\quad\nu_{2}(\mathbb{R}^{n}\backslash L)\geq 1-\delta.\]
\end{lemma}
\begin{proof}
Let $P$ be orthogonal projection from $\mathbb{R}^{n}$ onto the span of $z_{1},z_{2}$, where $z_{1}$ and $z_{2}$ are the means of $\nu_{1}$ and $\nu_{2}$, respectively. Let $\widetilde{\nu}_{1}:=(I-P)_{\#}\nu_{1}$ and $\widetilde{\nu}_{2}:=(I-P)_{\#}\nu_{2}$. We have
\begin{enumerate}[(1)]
\item $\widetilde{\nu}_{1}$ has mean $0$ and covariance $(I-P)\mathrm{Cov}(\nu_{1})(I-P)$
\item $\widetilde{\nu}_{2}$ has mean $0$ and covariance $(I-P)\mathrm{Cov}(\nu_{2})(I-P)$
\item $\alpha\widetilde{\nu}_{1}+(1-\alpha)\widetilde{\nu}_{2}$ is log-concave.
\end{enumerate}
Thus, $\|B_{\widetilde{\nu}_{1}}\|_{\mathrm{op}}=\|(I-P)\mathrm{Cov}(\nu_{1})(I-P)\|_{\mathrm{op}}\leq\|\mathrm{Cov}(\nu_{1})\|_{\mathrm{op}}\leq 1$, and similarly, $\|B_{\widetilde{\nu}_{2}}\|_{\mathrm{op}}\leq 1$. We also have
\begin{eqnarray*}
\|B_{\widetilde{\nu}_{1}}-B_{\widetilde{\nu}_{2}}\|_{\mathrm{F}}&=&
\|(I-P)(\mathrm{Cov}(\nu_{1})-\mathrm{Cov}(\nu_{2}))(I-P)\|_{\mathrm{F}}\\&\geq&
\|\mathrm{Cov}(\nu_{1})-\mathrm{Cov}(\nu_{2})\|_{\mathrm{F}}-4\|\mathrm{Cov}(\nu_{1})-\mathrm{Cov}(\nu_{2})\|_{\mathrm{op}}\\&\geq&
C_{2}\left(\ln\frac{1}{\delta}\right)^{2}\sqrt{\ln n}+4-4\|\mathrm{Cov}(\nu_{1})-\mathrm{Cov}(\nu_{2})\|_{\mathrm{op}}\\&\geq&
C_{2}\left(\ln\frac{1}{\delta}\right)^{2}\sqrt{\ln n},
\end{eqnarray*}
where the second step follows from Lemma \ref{gip} and the third and fourth steps follow from the assumptions on $\mathrm{Cov}(\nu_{1})$ and $\mathrm{Cov}(\nu_{2})$. Therefore, by Lemma \ref{smseparation}, there is a quadratic set $L\subset\mathbb{R}^{n}$ that depends only on $B_{\widetilde{\nu}_{1}}$ and $B_{\widetilde{\nu}_{2}}$ such that
\[\widetilde{\nu}_{1}(L)\geq 1-\delta\quad\text{and}\quad\widetilde{\nu}_{2}(\mathbb{R}^{n}\backslash L)\geq 1-\delta.\]
Thus,
\[\nu_{1}((I-P)^{-1}L)\geq1-\delta\quad\text{and}\quad\nu_{2}(\mathbb{R}^{n}\backslash(I-P)^{-1}L)\geq1-\delta,\]
where $(I-P)^{-1}L=\{x\in\mathbb{R}^{n}:\,(I-P)x\in L\}$.
Since $L$ is a quadratic set, $(I-P)^{-1}L$ is also a quadratic set. The result follows.
\end{proof}
\subsection{Completing the proof of the second main result}\label{secondmainproofsubsection3}
\begin{proof}[Proof of Theorem \ref{secondmain}]
Let $\mu_{1}$ and $\mu_{2}$ be log-concave probability measures on $\mathbb{R}^{n}$ such that $\mu_{j}$ has mean $z_{j}$ and covariance $\Sigma_{j}$ for $j=1,2$. Our goal is to construct a quadratic set $L\subset\mathbb{R}^{n}$ that depends only on $z_{1},z_{2},\Sigma_{1},\Sigma_{2}$ such that
\[\mu_{1}(L)\geq 1-\delta\quad\text{and}\quad\mu_{2}(\mathbb{R}^{n}\backslash L)\geq 1-\delta.\]

We first prove this assuming that $\Sigma_{1}+\Sigma_{2}$ is invertible.

Let $\nu_{1}=(\Sigma_{1}+\Sigma_{2})_{\#}^{-1/2}\mu_{1}$ and $\nu_{2}=(\Sigma_{1}+\Sigma_{2})_{\#}^{-1/2}\mu_{2}$. we have
\begin{enumerate}[(1)]
\item $\nu_{1}$ has mean $(\Sigma_{1}+\Sigma_{2})^{-1/2}z_{1}$ and covariance $(\Sigma_{1}+\Sigma_{2})^{-1/2}\Sigma_{1}(\Sigma_{1}+\Sigma_{2})^{-1/2}$
\item $\nu_{2}$ has mean $(\Sigma_{1}+\Sigma_{2})^{-1/2}z_{2}$ and covariance $(\Sigma_{1}+\Sigma_{2})^{-1/2}\Sigma_{2}(\Sigma_{1}+\Sigma_{2})^{-1/2}$
\item $\alpha\nu_{1}+(1-\alpha)\nu_{2}$ is log-concave.
\end{enumerate}
Thus, $\mathrm{Cov}(\nu_{1})+\mathrm{Cov}(\nu_{2})=I$, so $\|\mathrm{Cov}(\nu_{1})\|_{\mathrm{op}}\leq 1$ and $\|\mathrm{Cov}(\nu_{2})\|_{\mathrm{op}}\leq 1$.

We now consider two cases.

{\bf Case 1:} Assume
\[\|(\Sigma_{1}+\Sigma_{2})^{-1/2}(\Sigma_{1}-\Sigma_{2})(\Sigma_{1}+\Sigma_{2})^{-1/2}\|_{\mathrm{F}}\geq C_{2}\left(\ln\frac{1}{\delta}\right)^{2}\sqrt{\ln n}+4,\]
where $C_{2}>0$ is the absolute constant in Lemma \ref{covseparation}. In this case, we have
\[\|\mathrm{Cov}(\nu_{1})-\mathrm{Cov}(\nu_{2})\|_{\mathrm{F}}\geq C_{2}\left(\ln\frac{1}{\delta}\right)^{2}\sqrt{\ln n}+4,\]
so by Lemma \ref{covseparation}, there is a quadratic set $L\subset\mathbb{R}^{n}$ that depends only on the means and covariances of $\nu_{1},\nu_{2}$ such that
\[\nu_{1}(L)\geq 1-\delta\quad\text{and}\quad\nu_{2}(\mathbb{R}^{n}\backslash L)\geq 1-\delta.\]
Thus,
\[\mu_{1}((\Sigma_{1}+\Sigma_{2})^{1/2}L)\geq 1-\delta\quad\text{and}\quad\mu_{2}(\mathbb{R}^{n}\backslash[(\Sigma_{1}+\Sigma_{2})^{1/2}L])\geq 1-\delta.\]

{\bf Case 2:} Assume
\[\|(\Sigma_{1}+\Sigma_{2})^{-1/2}(z_{1}-z_{2})\|^{2}\geq C_{1}^{2}\left(\ln\frac{1}{\delta}\right)^{2},\]
where the $C_{1}>0$ is the absolute constant in Lemma \ref{meanseparation}. In this case, letting $z_{\nu_{1}}$ and $z_{\nu_{2}}$ be the means of $\nu_{1}$ and $\nu_{2}$, respectively, we have
\[\|z_{\nu_{1}}-z_{\nu_{2}}\|\geq C_{1}\ln\frac{1}{\delta},\]
so by Lemma \ref{meanseparation}, there is a set $L$ of the form $\{x\in\mathbb{R}^{n}:\,\langle x,w\rangle>t\}$ that depends only on $z_{\nu_{1}},z_{\nu_{2}}$ such that
\[\nu_{1}(L)\geq1-\delta\quad\text{and}\quad\nu_{2}(\mathbb{R}^{n}\backslash L)\geq 1-\delta.\]
Thus,
\[\mu_{1}((\Sigma_{1}+\Sigma_{2})^{1/2}L)\geq 1-\delta\quad\text{and}\quad\mu_{2}(\mathbb{R}^{n}\backslash[(\Sigma_{1}+\Sigma_{2})^{1/2}L])\geq 1-\delta.\]

Since the union of Case 1 and Case 2 covers the assumption (\ref{secondmaineq1}) in the statement of Theorem \ref{secondmain}, this completes the proof of the result when $\Sigma_{1}+\Sigma_{2}$ is invertible.

Next, we prove the result when $z_{2}=0$ but without assuming that $\Sigma_{1}+\Sigma_{2}$ is invertible.

If $z_{1}\notin\mathrm{ran}(\Sigma_{1}+\Sigma_{2})$, then $\mu_{1}$ is supported on $z_{1}+\mathrm{ran}(\Sigma_{1})\subset z_{1}+\mathrm{ran}(\Sigma_{1}+\Sigma_{2})$, whereas $\mu_{2}$ is supported on $\mathrm{ran}(\Sigma_{2})\subset\mathrm{ran}(\Sigma_{1}+\Sigma_{2})$. Letting $L=\{x\in\mathbb{R}^{n}:\,\langle Qx,z_{1}\rangle>0\}$ where $Q$ is the orthogonal projection $\mathbb{R}^{n}$ onto the orthogonal complement $\mathrm{ran}(\Sigma_{1}+\Sigma_{2})^{\perp}$, we have $\mu_{1}(L)=1$ and $\mu_{2}(\mathbb{R}^{n}\backslash L)=1$.

If $z_{1}\in\mathrm{ran}(\Sigma_{1}+\Sigma_{2})$ and (\ref{secondmaineq1}) holds, then the measures $\mu_{1}$ and $\mu_{2}$ are both supported on $\mathrm{ran}(\Sigma_{1}+\Sigma_{2})$. Moreover, $\Sigma_{1}+\Sigma_{2}$ is an invertible operator on $\mathrm{ran}(\Sigma_{1}+\Sigma_{2})$. Therefore, by restricting the ambient space $\mathbb{R}^{n}$ to $\mathrm{ran}(\Sigma_{1}+\Sigma_{2})$ and using what we have shown above, this completes the proof when $z_{2}=0$.

Finally, if $z_{2}\neq 0$, then we can translate the probability measures $\mu_{1}$ and $\mu_{2}$ by the vector $-z_{2}$ so that the translated $\mu_{2}$ has mean $0$ and apply what we have shown above.
\end{proof}

\appendix\section{Inverse operator}\label{inversesection}
Observe that for every positive semidefinite matrix $A\in\mathbb{R}^{n\times n}$, the restriction of the operator $A$ to $\mathrm{ran}(A)$ defines an invertible operator $A|_{\mathrm{ran}(A)}:\mathrm{ran}(A)\to\mathrm{ran}(A)$. Define the operator $A^{-1}:\mathrm{ran}(A)\to\mathrm{ran}(A)$ to be the inverse of $A|_{\mathrm{ran}(A)}$, i.e.,
\[A^{-1}:=(A|_{\mathrm{ran}(A)})^{-1},\]
and similarly, we define $A^{-1/2}:\mathrm{ran}(A)\to\mathrm{ran}(A)$ to be $A^{-1/2}:=(A|_{\mathrm{ran}(A)})^{-1/2}$.

For example, if $A=\begin{bmatrix}2&0&0\\0&3&0\\0&0&0\end{bmatrix}$, then $\mathrm{ran}(A)=\mathbb{R}^{2}$ and the operators $A^{-1}:\mathbb{R}^{2}\to\mathbb{R}^{2}$ and $A^{-1/2}:\mathbb{R}^{2}\to\mathbb{R}^{2}$ are given by
\[A^{-1}=\begin{bmatrix}\frac{1}{2}&0\\0&\frac{1}{3}\end{bmatrix}\quad\text{and}\quad A^{-1/2}=\begin{bmatrix}\frac{1}{\sqrt{2}}&0\\0&\frac{1}{\sqrt{2}}\end{bmatrix}.\]

For two positive semidefinite matrices $\Sigma_{1},\Sigma_{2}\in\mathbb{R}^{n\times n}$, observe that we always have $\mathrm{ker}(\Sigma_{1}+\Sigma_{2})=\mathrm{ker}(\Sigma_{1})\cap\mathrm{ker}(\Sigma_{2})$, and hence by taking orthogonal complements, we obtain
\[\mathrm{ran}(\Sigma_{1}+\Sigma_{2})=\mathrm{ran}(\Sigma_{1})+\mathrm{ran}(\Sigma_{2}).\]
We also have
\[\mathrm{ran}(\Sigma_{1}-\Sigma_{2})\subset\mathrm{ran}(\Sigma_{1})+\mathrm{ran}(\Sigma_{2})=\mathrm{ran}(\Sigma_{1}+\Sigma_{2}).\]
Thus, by restricting the operator $\Sigma_{1}-\Sigma_{2}$ to the subspace $\mathrm{ran}(\Sigma_{1}+\Sigma_{2})$, we have an operator on $(\Sigma_{1}-\Sigma_{2})|_{\mathrm{ran}(\Sigma_{1}+\Sigma_{2})}:\mathrm{ran}(\Sigma_{1}+\Sigma_{2})\to\mathrm{ran}(\Sigma_{1}+\Sigma_{2})$. Therefore, it makes sense to consider the following operator on $\mathrm{ran}(\Sigma_{1}+\Sigma_{2})$:
\[(\Sigma_{1}+\Sigma_{2})^{-1/2}(\Sigma_{1}-\Sigma_{2})(\Sigma_{1}+\Sigma_{2})^{-1/2},\]
which is the composition of $(\Sigma_{1}+\Sigma_{2})^{-1/2}$, $(\Sigma_{1}-\Sigma_{2})|_{\mathrm{ran}(\Sigma_{1}+\Sigma_{2})}$ and $(\Sigma_{1}+\Sigma_{2})^{-1/2}$.

For example, if $\Sigma_{1}=\begin{bmatrix}2&0&0\\0&1&0\\0&0&0\end{bmatrix}$ and $\Sigma_{2}=\begin{bmatrix}1&0&0\\0&0&0\\0&0&0\end{bmatrix}$, then $\mathrm{ran}(\Sigma_{1}+\Sigma_{2})=\mathbb{R}^{2}$ and the operator $(\Sigma_{1}+\Sigma_{2})^{-1/2}(\Sigma_{1}-\Sigma_{2})(\Sigma_{1}+\Sigma_{2})^{-1/2}:\mathbb{R}^{2}\to\mathbb{R}^{2}$ is given by
\[\begin{bmatrix}3&0\\0&1\end{bmatrix}^{-1/2}\begin{bmatrix}1&0\\0&1\end{bmatrix}\begin{bmatrix}3&0\\0&1\end{bmatrix}^{-1/2}=\begin{bmatrix}\frac{1}{3}&0\\0&1\end{bmatrix}.\]
\end{document}